\documentclass[12pt]{amsart}
\usepackage{amssymb}
\usepackage{amscd}
\usepackage{url}
\usepackage{graphics}
\usepackage{color}
\usepackage{tikz}
\usetikzlibrary[arrows,calc,through,backgrounds]
\usepackage{hyperref}

\numberwithin{equation}{section}

\theoremstyle{plain}

\newtheorem{theorem}[subsection]{Theorem}
\newtheorem{proposition}[subsection]{Proposition}
\newtheorem{lemma}[subsection]{Lemma}
\newtheorem{corollary}[subsection]{Corollary}

\theoremstyle{definition}

\newtheorem{definition}[subsection]{Definition}

\newcommand{\Lip}{\mathrm{Lip}}
\newcommand{\supp}{\textrm{supp }}
\newcommand{\E}{\mathbb{E}}
\newcommand{\Z}{\mathbb Z}
\newcommand{\back}{\backslash}
\newcommand{\gammaq}{\Gamma_0(q)}
\newcommand{\Hy}{\mathbb H}
\newcommand{\Q}{\mathbb Q}

\newcommand{\N}{\mathbb N}
\newcommand{\R}{\mathbb{R}}

\newcommand{\modu}{\textrm{ mod }}
\newcommand{\domainq}{\Gamma_0(q)\backslash \mathbb H}
\newcommand{\arccot}{\,\textrm{arccot}\,}

\begin{document}

\title{The horocycle flow at prime times}

\author{Peter Sarnak}
\address{Institute for Advanced Study\\
Einstein Road\\
     Princeton NJ 08540\\
     USA
}
\email{sarnak@math.princeton.edu}

\author{Adri\'an Ubis}
\address{Departamento de Matem\'aticas \\ Universidad Aut\'onoma de Madrid \\ Madrid 28049 \\ Spain
}
\email{adrian.ubis@uam.es}

\thanks{}

\subjclass{}

\begin{abstract}
We prove that the orbit of a non-periodic point at prime values of the horocycle flow in the modular surface is dense in a set of positive measure. For some special orbits we also prove that they are dense in the whole space---assuming the Ramanujan/Selberg conjectures for $GL_2/\Q$. In the process, we derive an effective version of Dani's Theorem for the (discrete) horocycle flow.
\end{abstract}

\maketitle

\section{Introduction}\label{introduction}

If $(X, T)$ is a dynamical system, for any $x\in X$ one can ask about the distribution of points $P_x =\{ T^p x: p \text{ prime} \}$ in the orbit $\theta_x=\{ T^n x : n\ge 1 \}$. For example if $X$ is finite then this is equivalent to Dirichlet's Theorem on primes in an arithmetic progression. If $(X, T)$ is ergodic, Bourgain \cite{bourgain} shows that for almost all $x$, $T^p x$ with $p$ prime, satisfies the Birkhoff ergodic Theorem and hence is equidistributed. If $(X, T)$ is  `chaotic', for example if it has a positive entropy then there may be many $x$'s for which $T^p x$ is poorly distributed in $\overline{\theta_x}$. For example if $T:[0,1]\to [0,1]$ is the doubling map $x\mapsto 2x$ then one can construct an explicit (in terms of its binary expansion) $\xi$ such that $\overline{\theta_{\xi}}=[0,1]$ but $T^p\xi\to 0$ as $p\to \infty$.

The setting in which one can hope for a regular behaviour on restricting to primes is that of unipotent orbits in a homogeneous space. Let $G$ be a connected Lie group, $\Gamma$ a lattice in $G$ and $u\in G$ an $Ad_G$ unipotent element, then Ratner's Theorem \cite{ratner} says that if $X=\Gamma\backslash G$ and $T: X\to X$ is given by
\begin{equation}\label{UnipotentFlowDef}
T(\Gamma g)= \Gamma g u,
\end{equation}
then $\overline{\theta_x}$, with $x=\Gamma g$, is homogeneous and the orbit $xu^n$, $n=1,2,\ldots$ is equidistributed in $\overline{\theta_x}$ w.r.t. an algebraic measure $d\mu_x$. In the case that $\mu_x$ is the normalized volume measure $d\mu_G$ on $X$ it is conjectured in \cite{gorodnik} that $\overline{P_x}=X$ and in fact that $xu^p$, $p=2,3,5,7,11,\ldots$ is equidistributed w.r.t. $d\mu_G$. Care should be taken in formulating this conjecture in the intermediate cases where $\overline{\theta_x}$ is not connected as there may be local congruence obstructions, but with the `obvious'  modifications this conjecture seems quite plausible. In intermediate cases where $\overline{\theta_x}$ is one of 
\begin{enumerate}
\item finite
\item a connected circle or more generally a torus
\item a connected nilmanifold $\Gamma \backslash N$,
\end{enumerate}
$\overline{P_x}$ and the behaviour of $xu^p$, $p=2,3,5,\ldots$ is understood. Case (i) requires no further comment while for (ii) it follows from Vinogradov's work that the points are equidistributed w.r.t  $dt$, the volume measure on the torus. The same is true for (iii)  as was shown recently by Green and Tao \cite{green-tao-equidistribution,green-tao-mobius}; in order to prove this, apart from using Vinogradov's methods they had to control sums of the type $\sum_n e(\alpha n [\beta n])$, which are similar to Weyl sums but behave in a more complex way.

Our purpose in this paper is to examine this problem in the basic case of $X=SL(2,\Z)\backslash SL(2,\R)$. According to Hedlund \cite{hedlund}, $\overline{\theta_x}$ is either finite, a closed horocycle of length $l$, $0<l<\infty$, or is all $X$. The first two cases correspond to (i) and (ii). In the last case we say that $x$ is \emph{generic}. By a theorem of Dani \cite{dani} the orbit $xu^n$, $n=1,2,3,\ldots$ is equidistributed in its closure w.r.t. one of the corresponding three types of algebraic measures. For $N\ge 1$ and $x\in X$ define the probability measure $\pi_{x,N}$ on $X$ by 
\begin{equation}\label{PrimeMeasure}
\pi_{x,N} := \frac{1}{\pi(N)} \sum_{p < N} \delta_{xu^p}
\end{equation}
where for $\xi\in X$, $\delta_{\xi}$ is the delta mass at $\xi$ and $\pi(N)$ is the number of primes less than $N$. We are interested in the weak limits $\nu_x$ of the $\pi_{x,N}$ as $N\to \infty$ (in the sense of integrating against continuous functions on the one-point compactification of $X$). If $x$ is generic then the conjecture is equivalent to saying that any such $\nu_x$ is $d\mu_G$. One can also allow $x$, the initial point of the orbit, to vary with $N$ in this analysis and in the measures in (\ref{PrimeMeasure}). Of special interest is the case $x=\Gamma g$
\[
g=H_N:=\begin{bmatrix}
 N^{-\frac 12 } & 0  \\
 0 & N^{\frac 12 }  \\
\end{bmatrix}
\hspace{10pt}
\mathrm{ and }
\hspace{10pt}
u=\begin{bmatrix}
 1 & 1  \\
 0 & 1  \\
\end{bmatrix},
\]
when $H_N u^j$, $0\le j\le N-1$ is a periodic orbit for $T$ of period $N$. These points are spread evenly on the unique closed horocycle in $X$ whose length is $N$. They also comprise a large piece of the Hecke points in $X$ corresponding to the Hecke correspondence of degree $N$
\[
C_N=\{ \frac{1}{\sqrt N } \begin{bmatrix}
a & b \\
c & d \\ 
\end{bmatrix}: ad=N, a, d>0,\, b \,\, \text{mod}\,  d \}.
\]

We can now state our main results. The first asserts that $\nu_x$ does not charge small sets with too much mass, that is $\nu_x$ is uniformly absolutely continuous with respect to $d\mu_G$. 

\begin{theorem}[Non-concentration at primes]\label{measure-control}
Let $x$ be generic and $\nu_x$ a weak limit of $\pi_{x,N}$, then
\[
d\nu_x \le 10\, d\mu_G.
\]
\end{theorem}

\emph{Remark:} As we said before, probably what truly happens is that $d\nu_x=d\mu_G$.
If on the other
hand we allow $n$ to vary not over primes but 
over almost primes, then the quantitative equidistribution that we
develop to prove Theorem \ref{measure-control} can be used together
with a lower bound sieve (see \cite{friedlander-iwaniec}, Chapter 12) to prove the density
of the orbit. More precisely let $x$ be generic, then the points
$T^n x$, as $n$ varies over numbers with at most 10 prime
factors, are dense in $X$.
 
As a consequence of Theorem \ref{measure-control} we deduce that $\overline{P_x}$ has to be big.

\begin{corollary}[Large closure for primes]\label{PrimeOrbitBig}
Let $x$ be generic, then 
\[
\mathrm{Vol} (\overline{P_x}) \ge \frac {1}{10},
\]
and if $U\subset X$ is an open set with $\mathrm{Vol}(U)>1-1/10$ then $xu^{p}\in U$, for a positive density of primes $p$.
\end{corollary}

In the case of `Hecke orbits' $P_{H_N}$, we can prove more;

\begin{theorem}[Prime Hecke orbits are dense]\label{hecke-control}
Let $\nu$ be a weak limit of the measures $\pi_{H_N,N}$. Assuming the Ramanujan/Selberg Conjectures concerning the automorphic spectrum of $GL_2/\mathbb Q$ (see for example the Appendix in \cite{sarnak-rankin}) we have
\[
\frac{1}{5} \, d\mu_G \le d\nu \le \frac 95 \, d\mu_G.
\]
\end{theorem}

Theorem \ref{hecke-control} has an application to a variant of Linnik's problem on projections of integral points on the level $1$ surface for quadratic forms in $4$-variables. Let $N\ge 1$ and denote by $M_N$ the set of $2\times 2$ matrices whose determinant equals $N$. Denote by $\pi$ the projection $A\mapsto \frac{1}{\sqrt N} A$ of $M_N(\R)$ onto $M_1(\R)$. Using their ergodic methods Linnik and Skubenko \cite{linnik-skubenko} show that the projection of the integer points $M_N(\Z)$ into $M_1(\R)$ become dense as $N\to \infty$. In quantitative form they show that for $U$ a (nice) compact subset of $M_1 (\R)$ 	
\begin{equation}\label{Linnik}
|\{ A\in M_N(\Z): \pi(A) \in U \} | \sim \sigma_1 (N)\mu (U)
\end{equation}
as $N\to \infty$, where $\mu(u)$ is the `Hardy-Littlewood' normalized Haar measure for $SL_2 (\R)$ on $M_1 (\R)$ and $\sigma_1(N)=\sum_{d\mid N} d$. \eqref{Linnik} can also be proved using Kloosterman's techniques in the circle method \cite{kloosterman} as well as using Hecke Correspondences in $GL_2$ as explained in \cite{sarnak-correspondences}. Using the last connection we establish the following Corollary  whose formulation is cleanest when $N$ is prime and which we assume is the case in the corollary. For
\[
A=\begin{bmatrix}
 a & b  \\
 c & d  \\
\end{bmatrix} \in M_N(\Z)
\]
we let $b_1(A)$ be the slope of the kernel of $A$ (which is a line) modulo $N$. That is if $N\mid a$ set $b_1(A)=\infty$ and otherwise $b_1(A)$ is the unique integer $0\le b_1<N$ satisfying $b_1\equiv \overline{a}b \mod N$ where $\overline a a \equiv 1 \mod N$. 
\begin{corollary}\label{HeckeDensity}
Assume the Ramanujan/Selberg Conjectures for $GL_2$. Let $U$ be a (nice) compact subset of $M_1(\R)$ and $\epsilon>0$. Then for $N$  prime sufficiently large
\[
\frac{1/5-\epsilon}{\log N} \le \frac{ |\{A\in M_N(\Z): \pi(A)\in U, \, b_1(A) \text{ is prime } \}|}{ |\{ A\in M_N(\Z): \pi(A)\in U \}| } \le \frac{9/5+\epsilon}{\log N}.
\]
In particular the projections of the points $A\in M_N(\Z)$ with $b_1(A)$ prime, become dense in $M_1(\R)$.
\end{corollary}

We end the introduction with an outline of the contents of the sections and of the techniques that we use. The proofs of Theorems \ref{measure-control} and \ref{hecke-control} use of sieve methods. These reduce sums over primes $\sum_{p\le N} f(xu^p)$, to the study of linear sums over progressions $\sum_{n\le N/d} f(xu^{nd})$, and bilinear sums $\sum_{n\le \min(N/d_1, N/d_2 )} f_1(xu^{nd_1})f_2(xu^{nd_2})$. A critical point in the analysis is to allow $d$ to be as large as possible, this is measured by the level of distribution $\alpha$; $d\le N^{\alpha}$ (respectively $\max(d_1, d_2)\le N^{\alpha}$ ).  The first type of sums are connected with equidistribution in $(X,T)$ and the second type with joinings of $(X,T^{d_1})$ with $(X,T^{d_2})$.

The  effective rate of equidistribution of long pieces of unipotent orbits has been studied in the case of general compact quotients $\Gamma\backslash SL(2,\R)$.  For continuous such orbits this is due to Burger \cite{burger} while for discrete ones to Venkatesh \cite{venkatesh}. Both make use of the spectral gap in the decomposition of $SL(2,\R)$ acting by translations on $L^2 (\Gamma \backslash SL(2,\R))$. One can quantify Venkatesh's method to obtain a positive level of distribution for the linear sums and then follow the analysis in the proof of Theorem \ref{measure-control} to obtain the analogues of Theorem \ref{measure-control} and Corollary \ref{PrimeOrbitBig} for such $\Gamma\,$'s (the constant $10$ is replaced by a number depending on the spectral gap).

For $\Gamma\backslash SL(2,\R)$ noncompact but of finite volume, due to the existence of periodic orbits of the horocycle flow one cannot formulate a simple uniform rate of equidistribution in Dani's Theorem. In a preprint \cite{strombergsson-integral} A. Str\"ombergsson gives an effective version of Dani's Theorem for continuous orbits in terms of the excursion rate of geodesics; he uses Burger's approach.

We only learned of \cite{strombergsson-integral} after completing our formulation and treatment of an effective Dani theorem for the continuous flows, see Theorems \ref{implicit-dani-continuous} and \ref{dani-continuous} in section \ref{effective-equid-non-discrete}. One of the main results of this paper is an effective Dani Theorem for discrete unipotent orbits of $(X,T^s)$, see Theorems \ref{dani-discrete} and \ref{effective-equid}. These give a quantitative equidistribution w.r.t. algebraic measures of long pieces of such orbits and allowing $s$ to be large. This discrete case is quite a bit more complex both in its formulation and its proof. It requires a series of basic Lemmas (see section \ref{horocycle-approximation}) which use the action of $SL_2(\Z)$ to give quantitative approximations of pieces of horocycle orbits by periodic horocycles, much like the approximation of reals by rationals in the theory of diophantine approximation. Critical to this analysis are various parameters associated with a given piece of horocycle orbit. The resulting approximations allow us to approach the level of distribution sums by taking $f, f_1, f_2$ to be automorphic forms and expanding them in Fourier series in the cusp. The burden of the analysis is in this way thrown onto the Fourier coefficients of these forms. This leads us to Theorem \ref{dani-discrete} which gives a suitable level of distribution in the linear sums and with which we can apply an upper bound sieve (Brun, Selberg) and deduce Theorem \ref{measure-control} and Corollary \ref{PrimeOrbitBig}.

Theorem \ref{hecke-control} involves a lower bound sieve and in particular a level of distribution for bilinear sums. This is naturally connected with effective equidistribution of $1$-parameter unipotent orbits in $(SL(2,\Z)\backslash SL(2,\R)) \times (SL(2,\Z)\backslash SL(2,\R))$ which is a well known open problem since Ratner's paper \cite{ratner-joinings}. For the special Hecke points that are taken in Theorem \ref{hecke-control} our Fourier expansion approach converts the bilinear sums into sums of products of shifted coefficients of these automorphic forms (``Shifted convolution''). In Section \ref{automorphic-forms} we review the spectral approach to this well studied problem; in particular we use the recent treatments in \cite{blomer-harcos,blomer-harcos-adeles} which are both convenient for our application and also allow for a critical improvement over \cite{sarnak-rankin} in the level aspect. Proposition \ref{period-bound-proposition} gives a slight improvement over \cite{blomer-harcos-adeles} and also \cite{pi}, in this $q$ aspect, and it is optimal under the Ramanujan/Selberg Conjectures. Concerning the linear sums for these Hecke orbits we establish a level of $1/2$ (see the discussion at the end of section \ref{density-hecke-orbit}). This is the optimal level that can be proved by automorphic form/spectral methods. To analyze the sum over primes we use the sieve developed by Duke-Friedlander-Iwaniec \cite{duke-friedlander-iwaniec}. For an asymptotics for the sum over primes (i. e. a ``prime number theorem'') they require a level of distribution of $1/3$ for the bilinear sums, given the level of $1/2$ that we have for the linear sums. Using the best bounds towards Ramanujan/Selberg Conjecture for $GL_2/\mathbb Q$ we establish a level of $\alpha=3/19$ for these bilinear sums. This falls short of the $1/3$ mark as well as the $1/5$ mark which is needed to get a lower bound in the sum over primes. However assuming the Ramanujan/Selberg Conjecture this does give a strong enough level of distribution to deduce Theorem \ref{hecke-control}.

\section{Horocycle Approximation}\label{horocycle-approximation}

The group $G=SL(2,\R)$ can be parametrized through its Lie algebra. The Lie algebra $\mathfrak g$ are the $2\times 2$ real matrices of zero trace. In this way, the so-called Iwasawa parametrization $g=h(x)a(y)k(\theta)$ with $x,y,\theta\in \R$ and
\[
h(x)=\begin{bmatrix}
 1 & x  \\
 0 & 1  \\
\end{bmatrix}, \hspace{20pt}
a(y)=\begin{bmatrix}
 y^{\frac 12} & 0  \\
 0 & y^{-\frac 12}  \\
\end{bmatrix}, \hspace{20pt}
k(\theta)=\begin{bmatrix}
 \cos \theta & \sin \theta  \\
 -\sin \theta & \cos \theta  \\
\end{bmatrix}.
\]
corresponds to the Lie algebra basis
\begin{equation}\label{LieBasis}
R=\begin{bmatrix}
0 & 1 \\
0 & 0
\end{bmatrix},
\hspace{20pt}
H=\begin{bmatrix}
1 & 0 \\
0 & -1
\end{bmatrix},
\hspace{20pt}
V=\begin{bmatrix}
0 & 1 \\
-1 & 0
\end{bmatrix}
\end{equation}
in the sense that $h(x)=\exp(xR)$, $a(e^{2u})=\exp(uH)$ and $k(\theta)=\exp(\theta V)$.
Moreover it is unique when restricting $\theta$ to $[-\pi,\pi)$.

We can explicitly define a left $G$-invariant metric on $G$ as
\[
d_G(g,h)=\inf\left\{\sum_{i=0}^{n-1} \psi(x_i, x_{i+1}) : x_0,\ldots , x_n\in G; x_0=g; x_n=h \right\}
\]
with $\psi(x,y)=\min(\|x^{-1}y-I\|,\|y^{-1}x-I\|)$ and $\|\cdot \|$ any norm. Let us fix the norm 
\[
\left\|
\begin{bmatrix}
a & b \\
c & d 
\end{bmatrix}
\right\|=\sqrt{2a^2+(b+c)^2+4c^2+2d^2}.
\]
for concreteness. Then, we can describe the metric in terms of the Iwasawa parametrization as
\[
d_G s^2= \frac{dx^2+dy^2}{y^2}+d\theta^2
\]
We can also write any Haar measure in $G$ as a multiple of
\[
d\mu_G= \frac{dx }{y} \, \frac{dy}{y} \, \frac{d\theta}{\pi}.
\]
Since $G$ is unimodular, this measure is both left and right $G$-invariant.

By sending $(x+iy,\theta)$ to $h(x)a(y)k(\theta/2)$ we see that $\{\pm I\}\backslash G$ endowed with this metric is isometric to the unit tangent bundle $T_1 \mathbb{H}$ of the Poincar\'e upper half-plane $\mathbb{H}$ with the metric $ds^2=y^{-2}(dx^2+dy^2)$. We shall use both notations to refer to an element of $\{\pm I\}\backslash G$, and we shall even use $x+iy$ as a shorthand for $(x+iy,0)$. In this way, we can express multiplication in $\{\pm I\}\backslash G$ as
\begin{equation}\label{gammaMultiplication}
\begin{bmatrix}
 a & b  \\
 c & d  \\
\end{bmatrix}
(z,\theta)=\left(\frac{az+b}{cz+d},\theta-2\arg (cz+d)\right).
\end{equation}
Now, we define the discrete \emph{horocycle flow} at distance $s$ as the transformation $g\mapsto gh(s)$. The name comes from the fact that a horocycle is a circle in $\mathbb{H}$ tangent to $\partial\mathbb{H}$, and the horocycle flow sends a point with tangent vector pointing towards the center of the horocycle $S$ to the unique point at distance $s$ ``to the right'' whose tangent vector also points to the center of $S$. In terms of our parametrization, we can write 
\begin{equation}\label{ActionHorocycle}
gh(t+\cot \theta)=h(\alpha-\frac{R t}{t^2+1})a( \frac{R}{t^2+1})k(-\arccot t)
\end{equation}
for $g=h(x)a(y)k(\theta)$, where 
\[
R=y(\sin\theta)^{-2}
\]
is the diameter of the horocycle and 
\[
\alpha=x-yW
\]
its point of tangency with $\partial \mathbb{H}$, where $W=\cot \theta$ (see Figure \ref{horocycle-flow}).

 %%%%%%%%%%%%%%%%%%%%%%%%
 
 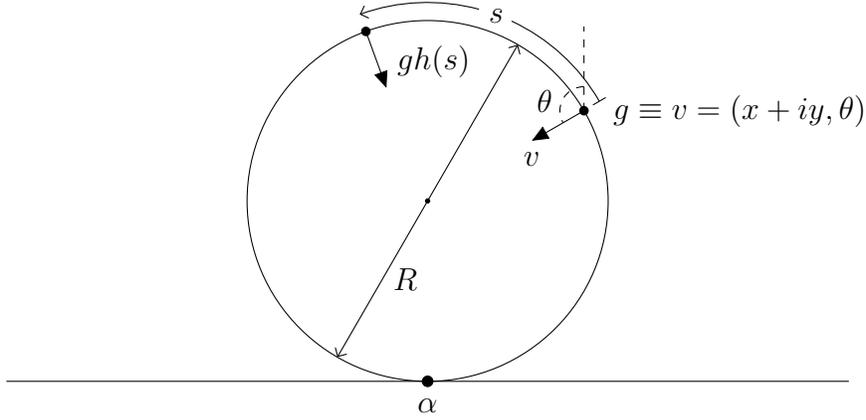
\begin{figure}
   \caption{The horocycle flow}
   \label{horocycle-flow}
   \centering
 
 \ 
  
 \begin{tikzpicture}[scale=0.8]
 
% % axis
% \draw [name path=axis] (-7,0) --  (7,0);
% 
% % horocycle
% \draw [name path=horocycle] (0,3) circle (3);
% 
% % tangency point
% \filldraw [name intersections={of=axis and horocycle, by={a}}] circle (2.5pt) node[below=2.5pt] {$\alpha$};
 
 % new axis
 \draw (-7,0) -- (7,0);
 
 % new horocycle
 \draw (0,3) circle (3);
 
 % new tangency point;
 \filldraw (0,0) circle (2.5pt) node[below=2.5pt] {$\alpha$};
 
 % center
 \filldraw (0,3) circle (1pt);
 
 % radius
 \draw [-angle 90] (0,3) -- node[right] {$R$} +(-120:3);
 
 % diameter
 \draw [-angle 90] (0,3) --  +(60:3);

 % g arrow
 \draw (0,3) ++(30:3)  [-triangle 45] -- ++(210:1) node[below] {$v$};
 
 % g point
 \filldraw (0,3) ++(30:3)  circle (2pt) node[right=7pt] {$g\equiv v =(x+iy,\theta)$};
 
 % angle axis
 \draw[dashed] (0,3) ++(30:3) -- ++(90:1.4);
 
 % angle 
 \draw[dashed] (0,3) ++(30:3) ++(90:0.4) arc (90:210:0.4) node[above left] {$\theta$};
 
 % gh(s) arrow
 \draw (0,3) ++(110:3)  [-triangle 45] -- ++(-70:1) node[above right] {$gh(s)$};
 
 % gh(s) point
 \filldraw (0,3) ++(110:3)  circle (2pt);

 % s coordinates
 %\coordinate [label=center:{$s$}] (s) at ($ (0,3) + (70:3.3) $);
 
% new s coordinates
\draw (1.14,6.08) node {$s$};

 % distance s
 \draw [|-] (0,3) ++(30:3.3)  arc  (30:65:3.3);
 \draw [angle 90-] (0,3) ++(110:3.3)  arc  (110:75:3.3);
 
 \end{tikzpicture}
 
 \end{figure}
 %%%%%%%%%%%%%%%%%%%%%%%%

Now, we consider the homogeneous space $X=\Gamma\backslash G$, with the metric induced from the one in $G$, namely $d_X(\Gamma g,\Gamma h)=\min_{\gamma\in\Gamma} d_G(\gamma g,h)$. 
Then, we have that $X$ is isometric to $T_1(\Gamma\backslash\mathbb{H})$.  So, considering the identification $g=(z,\theta)$, we can set
\[
D_X=\{(z,\theta): |z|\ge 1, \, -1/2\le \Re z \le 1/2, -\pi \le \theta \le \pi\}
\]
as a fundamental domain for $X$.

We also have that the horocycle flow in $G$ descends to $X$, and when doing so its behaviour becomes more complex. As we noted in the introduction, Dani proved that for any $\xi\in X$ the orbit generated by the horocycle flow, $\{\xi h(s)^n:n=0,1,2,\ldots\}$, is dense in either
\begin{enumerate}
\item a discrete periodic subset of a closed horocycle
\item a closed horocycle
\item the whole space, 
\end{enumerate}
and we can explicitly state which possibility happens in terms of $\xi=\Gamma g$: (i) for $\alpha$ and $sR^{-1}$ rational numbers, (ii) for $\alpha$ rational and $sR^{-1}$ irrational and (iii) for $\alpha$ irrational. Moreover, in each case the orbit becomes equidistributed in its closure w.r.t. the algebraic probability measure supported there.

If one considers the continuous version of the horocycle flow, $\{\xi h(t): t\in \R_{\ge 0}\}$, only the possibilities (ii) and (iii) can occur, depending just on the rationality of $\alpha$.

Our purpose in sections \ref{effective-equid-non-discrete} and \ref{section-discrete-measures} is to analyze which of the possibilities is the ``nearest'' for a finite orbit
\[
\{\xi h(s)^n: 0\le sn \le T \};
\]
or $\{\xi h(t): t\in [0,T]\}$ in the continuous case, in terms of the parameters $\alpha$, $sR^{-1}$ and $T$. In preparation for that, we define some quantities associated to the piece of horocycle $\{ \xi h(t): t\in [0,T]\}$, which will be useful for applying spectral theory, and moreover allow us to decide between (i), (ii) and (iii) for $\xi$. 
Let $Y_T(g)$ the ``Euclidean distance'' from the piece of horocyle $P_{g,T} =\{gh(t): t\in [0,T] \}$ to the border $\partial \mathbb{H}$, namely
\[
Y_T(g)=\inf\{y:h(x)a(y)k(\theta)\in P_{g,T}\},
\]
which coincides with the $y$ associated two one of the extremes of the piece. Due to \eqref{ActionHorocycle} we have
\begin{equation}\label{ytsize}
Y_T(h(x)a(y)k(\theta))\asymp \min(y,\frac R{T^2}),
\end{equation}
where the symbol $\asymp$ is defined as follows.

\begin{definition}[Notation for bounds] 
We shall use the notation $f=O(g)$ or $f\ll g$ meaning $|f|\le C |g|$ for some constant $C>0$; we shall also write $f\asymp g$ as a substitute for $f\ll g\ll f$. Finally, we shall use the notation $f<g^{O(1)}$ and $f<g^{-O(1)}$, with $g>0$, meaning $|f|<g^{C}$ and $|f|<(g^{-1})^C$ respectively, for some constant $C>0$. The implicit constant $C$ will not depend on any other variable unless in a statement that contains an implication of the kind ``if $f_1=O(g_1)$ then $f_2=O(g_2)$''; in that case the constant implicit in $O(g_2)$ depends on the one in $O(g_1)$. On the other hand, whenever we use the notation $g^{o(1)}$ or $g^{-1/|o(1)|}$, we mean that the function in $o(1)$ depends just on $g$ and goes to zero as $g\to \infty$.
\end{definition}

Now, the key concept is the following.
\begin{definition}[Fundamental period] Let $g\in G$ and $T\ge 1$. We define the \emph{fundamental period} of $\xi=\Gamma g$ at distance $T$ as $y_T^{-1}$, where
\[
y_T=y_T(\Gamma g )=\sup \{Y_T(\gamma g):\gamma\in \Gamma\}.
\]
\end{definition}

The point of this definition is that we want to approximate our piece of horocycle by a closed one. Now, a closed horocycle has the shape
$
\Gamma a(y)H
$
with $H$ the closed subgroup $H=\{h(t):t\in \R\}$; the period of this closed horocycle is $y^{-1}$---in the sense that $\Gamma a(y)h(\cdot)$ is a periodic function of period $y^{-1}$. We shall see that the closed horocycle of period $y_T^{-1}$ is near to our original piece of horocycle $\{\Gamma g h(t): t\in [0,T]\}$.

From the discontinuity of the $\Gamma$ action one can deduce that the supremum in the definition of $y_T$ is actually reached.  Moreover, for $g\in U$ an open dense subset of $G$ one can assure that this happens for a unique point $g_T=h(x)a(*)k(*)$ with $-1/2<x\le 1/2$, of the shape either $\gamma g$ or $\gamma g h(T)$ for some $\gamma\in \Gamma$. This defines the key parameters $\theta_T$, $\alpha_T$, $y_T$ and $W_T$ associated to $g_T$ (see Figure \ref{highest-position}).

 %%%%%%%%%%%%%%%%%%%%%%%%%%%%%%
 
 \begin{figure}
   \caption{Piece of horocycle in highest position}
   \label{highest-position}
   \centering
 
 \ 
 
 \begin{tikzpicture}[scale=0.9]
 
% % axis
% \draw [name path=axis] (-7,0) --  (7,0);

% new axis
\draw (-7,0) -- (7,0);
 
 % g_T point coordinates
 \coordinate (gT) at ($ (-5,0) +(90:25)+ (-80:25) $);
 
 % g_T 
 \filldraw (gT) circle (2pt) node[right=25pt] {$g_T=(x_T+iy_T,\theta_T)$} ;
 
 % v_T
 \draw [-triangle 45] (gT) -- ++(100:1) node[left] {$v_T$};
 
%  % angle axis
%  \draw[loosely dashdotted] (gT) -- ++(90:1.4);
 
% new angle axis
\draw[dashed] (gT) -- ++(90:1.4);

 % angle 
 \draw[dashed] (gT) ++(90:0.5) arc (90:100:0.5) node[above right=2.5pt] {$\theta_T$};
 
 % g_T h(t) point coordinates
 \coordinate (gTt) at ($ (-5,0) +(90:25)+ (-71:25) $);
 
 % g_T h(t) 
 \filldraw (gTt) circle (2pt) node[right=8pt] {$g_T h(t) \, \, \, \, t\in [0,T]$} ;
 
 % v_T h(T)
 \draw [-triangle 45] (gTt) -- ++(109:1);
 
 % g_T h(T) point coordinates
 \coordinate (gThT) at ($ (-5,0) +(90:25)+ (-66:25) $);
 
 % g_T h(T) 
 \filldraw (gThT) circle (2pt) node[right=8pt] {$g_T h(T)$} ;
 
 % v_T h(T)
 \draw [-triangle 45] (gThT) -- ++(114:1);
 
 % piece
 \draw (gT) arc (-80:-66:25);
 
 % horocycle long
 \draw[dashed] (-5,0) arc (-90:-80:25);
 \draw[dashed] (gThT) arc (-66:-61.2:25);
 
 % horocycle short
 \draw[dashed] (-5,0) arc (-90:-94.7:25);
 
 % tangency point
 \filldraw (-5,0) circle (2.5pt) node[below=2.5pt] {$\alpha_T$};
 
 % W_T
 \draw  (-5,2) node {$W_T=\cot \theta_T$};
 
 \end{tikzpicture}
 
 \end{figure}
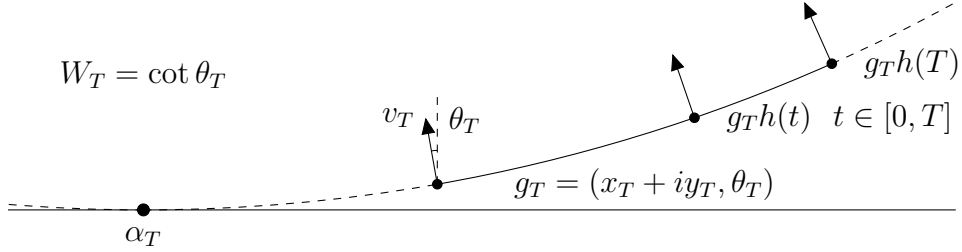
 
 %%%%%%%%%%%%%%%%%%%%%%%%%%%%%%%%

Let us define the following equivalence relation in the space of pieces of horocycles of length $T$: $P\sim P'$ if there exists $\gamma\in\Gamma$ such that $P'=\gamma P$ as sets. We can identify a piece $P$ with its point $g$ which is nearest to $\partial \mathbb H$---in case that both extremes are at the same distance from the border, we choose the left one.  In this way, we can see the space of pieces of horocycles of length $T$ as a subset of $PSL_2(\mathbb R)$. Then, we have just showed that
\[
D_{X,T} =\overline{ \left\{g_T:g\in U \cap D_X \right\} };
\]
is a fundamental domain for this equivalence relation. We give an explicit description of this fundamental domain beginning by realizing $y_T$ arithmetically; for that purpose we make the following definitions.

\begin{definition}[Torus distance] 
Let $\alpha\in \R$. We define the \emph{integral part} of $\alpha$, and write $[\alpha]$, as the nearest integer to $\alpha$. We also define its \emph{fractional part} as $\{\alpha\}=\alpha-[\alpha]$. Finally, we define $\|\alpha\|$ as the absolute value of $\{\alpha\}$.
\end{definition}

\begin{definition}[Rational approximation]
Let $\alpha\in \R$ and $U>0$. We define
\[
\kappa_U(\alpha)=\min\{m\in \N: \|m\alpha\|\le U^{-1}\}.
\]
\end{definition}
We have the following property regarding the previous definition.
\begin{lemma}\label{diophantine_property}
Let $\alpha\in\R$, $U> 0$ and $q\in \N$. If $\|q\alpha\|\le 1/U$ then either $\|q\alpha\|$ is a multiple of $\|\kappa_U(\alpha)\alpha\|$ or $q\ge U/2$. 
\end{lemma}
\begin{proof}
We can assume $U>2$. Writing $k=\kappa_U(\alpha)$, we have
\[
|\alpha-\frac{a}{q}|\le\frac 1{Uq}  \qquad   |\alpha-\frac{b}{k}|\le
\frac{1}{Uk}
\]
for some integers $a,b$, $(b,k)=1$. Thus, either $a/q=b/k$ or
\[
\frac{1}{qk}\le |\frac aq -\frac bk|\le\frac{1}{Uq}+\frac 1{Uk}\le \frac{2}{Uk}
\]
which implies $q\ge U/2$. In the former case, for some $\lambda\in \N$ we have $1/2>|q\alpha-a|=\lambda|k\alpha-b|$, so that $\|q\alpha\|=|q\alpha-a|$.
\end{proof}
From \eqref{ytsize} we know that $y_T(\Gamma g)$ is always $\gg (2T)^{-2}$. It is natural that we can improve on that by translating $g$ by different elements $\gamma\in \Gamma$. The following result reflects as far as we can go.

\begin{lemma}[Period realization]\label{fundamental}
Let $T\ge 5$. For any $g$ with $y=y(g)\gg 1$ we have
\[
y_T(\Gamma g) \asymp \min(y,R T^{-2})+T^{-1}\min\left(\frac {U}{\kappa_U(\alpha) }, \frac{U^{-1}}{\|\kappa_U(\alpha) \alpha\|}\right)^2.
\]
with $U=(T/R)^{1/2}$. Therefore, for any $g\in G$ we have
\[
y_T(\Gamma g)\gg T^{-1}.
\]
\end{lemma}

\begin{proof}
We can rewrite the $\Gamma$ action (\ref{gammaMultiplication}) as
\[
y_{\gamma}=\frac{1}{R}\frac{1}{c^2+2 c\,\epsilon \cos\theta+\epsilon^2}, \hspace{20pt} R_{\gamma}=\frac{R}{(c\alpha+d)^2},
\]
where $\gamma$ is the matrix there---with $R_{\gamma}$ and $y_{\gamma}$ the corresponding parameters associated to $\gamma g$---and $\epsilon=\pm (yR_{\gamma})^{-1/2}$, the sign given by the one of $(c\alpha+d)\sin\theta$.  Let us first treat the case $R\ge T$. We want to show that $y_T \asymp \min(y,RT^{-2})$. If $y\le RT^{-2}$, since $y\gg 1$ it is clear that $y_T(\Gamma g)\asymp y$. If $y> RT^{-2}$, let us suppose that $y_T>C RT^{-2}$ for a large constant $C$. If $g_T$ equals either $\gamma g$ or $\gamma gh(T)$, we deduce from \eqref{ytsize} that $R_{\gamma}\gg C R$, so that $c\neq 0$ and $d=-[c\alpha]$ in the definition of $R_{\gamma}$. Therefore
\[
|y_{\gamma}-\frac{1}{c^2 R}|\ll \frac{1}{c^2R}(R_{\gamma}y)^{-1/2},
\]
and since $R_{\gamma}y\gg C R y \gg CT y \gg CT \gg C$ we get that $y_{\gamma}\ll c^{-2} R^{-1}\ll RT^{-2}$ which is in contradiction with our assumption.

Now let us treat the case $T\ge R$. Choosing $c=\kappa_U(\alpha)$ and $d=-[\kappa_U(\alpha)\alpha]$ in the previous formulas, we have $R_{\gamma}=R\|\kappa_U(\alpha)\alpha\|^{-2}\ge T$ and then $y_{\gamma}\asymp \kappa_U(\alpha)^{-2}R^{-1}$.
Considering (\ref{ytsize}) this clearly implies
\[
Y_{T}(\gamma g)\asymp \min(\frac 1{\kappa_U(\alpha)^2 R}, \frac{R}{T^2 \|\kappa_U(\alpha)\alpha\|^2}).
\]
which equals the second term in the sum of the Lemma's statement. Since $y_T(\Gamma g)\ge Y_{T}(\gamma g)$, it only remains to prove that $y_T(\Gamma g)\le C  Y_T (\gamma g) $ for some constant $C>1$. Let us suppose this is not the case; then there exists $\gamma_*\in \Gamma$ such that $Y_T (\gamma_* g)> CY_T (\gamma g)$ and since $Y_T (\gamma g) \gg T^{-1}$ by \eqref{ytsize} it follows that
$
\|c_*\alpha\|< U^{-1}
$
which by definition of $\kappa_U(\alpha)$ implies $c_*\ge \kappa_U(\alpha)$. Also we could repeat the previous reasoning to show that 
\[
Y_{T}(\gamma_* g)\asymp \min(\frac 1{c_*^2 R}, \frac{R}{T^2 \|c_*\alpha\|^2}).
\]
Finally,  by applying Lemma \ref{diophantine_property} with $q=c_*$ we have $Y_T(\gamma_* g)=O(Y_T (\gamma g))$ which is a contradiction. 
\end{proof}

Can we get something better than the bound $O(T)$ for the fundamental period? Not in general, think for instance of the case $g=a(T^{-1})$: then the piece of horocyle is just the closed horocyle of length $T$ and $y_T=T^{-1}$ in this case. This is not a coincidence, as the following result shows.

\begin{lemma}[Domain description]\label{FundDomainSubset}
For any $T\ge 2$ and $c>0$, we define the set
\[
B_c(T)=\{h(x)a(y)k(\theta): -1/2\le x \le 1/2, y^{-1}<Tc^{-1}, |\theta|<T^{-1}c^{-1} \}.
\]
We have that
\[
B_{c_1}(T)\subset  D_{X,T} \subset B_{c_2} (T) 
\]
for some positive constants $c_1, c_2$ and any $T\ge 2$.
\end{lemma}
\begin{proof}
The inclusion $D_{X,T} \subset B_{c_2}(T)$ comes just from Lemma \ref{fundamental} and the fact that by (\ref{ActionHorocycle}) the lowest point in any piece of horocycle of length $T$ has $\theta=O(T^{-1})$.

Let $g=h(x)a(y)k(\theta)\in B_{c_1}(T)$ for some large $c_1$. Suppose that $g\not \in D_{X,T}$, then there exists $\gamma\in \Gamma$ with $c=c_{\gamma}\neq 0$ such that $\gamma g\in D_{X,T}$. So, since $g$ and $g'=h(x')a(y')k(\theta'):=gh(T)$
are in $P_{g,T}$, by Lemma \ref{fundamental} we have
\[
\min(a(\gamma g),a(\gamma g')) = y_T(g) \ge \frac{\epsilon}T
\]
for some $\epsilon>0$. On the other hand, one can check that $|x'-x|\gg yT$ and $y'\asymp y$. Therefore, by (\ref{gammaMultiplication}) we have
\[
\min(a(\gamma g),a(\gamma g'))\ll \frac{y}{\max(|cx+d|,|cx'+d|)^2}\ll \frac{y}{|c|^2|x-x'|^2}\ll \frac{1}{yT^2}
\]
which is $O(\frac{1}{c_1 T})$, giving a contradiction for $c_1$ large enough.
\end{proof}

\emph{Remark:} This lemma implies that in any case, a large part of the piece of horocycle lies at height $O(y_T)$, thus showing that it is near to the closed horocycle of period $y_T^{-1}$. Moreover, it says that the fundamental domain is essentially
\[
|x|\le 1/2, \hspace{20pt}    y^{-1}\ll T \ll  |W|
\]
with $y^{-1}$ describing the period of the associated closed horocyle and $W=\cot\theta$ measuring the distance to it; $\theta=0$ being the extreme case in which the piece is actually a closed horocycle.

From \eqref{ActionHorocycle} we can write
\[
gh(t+\cot \theta)=h(\alpha-Rt^{-1})a(Rt^{-2})+O(t^{-1})
\]
meaning that both points are at distance $O(|t|^{-1})$ in $G$. Since $g_T$ can be either $\gamma g$ or $\gamma g h(T)$, we can always parametrize any piece of horocycle of length $T$ as
\begin{equation}\label{parametrization-piece-flat}
h\left(\alpha_T + \frac{y_T W_T}{1\pm t/W_T}\right)a\left(\frac {y_T}{(1\pm t/W_T)^2}\right)+O\left(\frac{W_T^{-1}}{1\pm t/W_T}\right)  \hspace{20pt} t\in [0,T].
\end{equation}

As we said before, we shall understand the piece of horocycle in terms of the parameters $\alpha_T,y_T$ and $W_T$ (and $s$ in the discrete case). But we are mainly interested in a fixed $\Gamma g$ and letting $s$ change, and in that situation we will be able to express the results just in terms of $\alpha$ and $sR^{-1}$. To do that, we need to relate both kind of conditions. First we write the necessary result for the continuous orbit. In order to read the following result, it is convenient to keep in mind that for any $g\in D_X$ with $y(g)<\delta^{-1}$ either $g$ or $\begin{bmatrix}0 & 1 \\ -1 & 0  \end{bmatrix}g$ satisfies that its coefficients $(R,\alpha,y^{-1})$ are bounded by $O(\delta^{-3})$.

\begin{lemma}[Fundamental period and continuous Dani]\label{period-made-explicit-continuous}
Let $0<\delta<1/2$ and $g_0\in G$ with coefficients  $(R,\alpha, y^{-1})$ bounded by $\delta^{-1}$. Then, $y_T^{-1}<\delta^{-O(1)}$ if and only if there exists a positive integer $q<\delta^{-O(1)}$ such that $\|q\alpha\|<\delta^{-O(1)}T^{-1}$. 
\end{lemma}

\begin{remark}
For our application to orbits at prime values we will always have $\delta^{-1}=(\log T)^A$ for some constant $A>0$ in this lemma as well as in later statements. Moreover, throughout the whole paper we can assume that $(\delta^{-1})^c<T$ for some large constant $c>1$, because otherwise the results are trivial.
\end{remark}

\begin{proof}
Since we always have $y_T^{-1}\ll T$, we can assume $\delta^{-O(1)}\ll T$ in the proof. One can check that the constants implicit in the statement of Lemma \ref{fundamental} have a polynomial dependency on the constant in $y\gg 1$. Thus, applying it to our case, since $y>\delta$, we will get  $\delta^{-O(1)}$ as implicit constants. Therefore, since $\delta<R,y<\delta^{-1}$, we have that Lemma \ref{period-made-explicit-continuous} is a direct consequence of Lemma \ref{fundamental} and Lemma \ref{diophantine_property}.
\end{proof}

Now we write the analogous result for a discrete orbit. Probably it is better to skip it on first reading, at least until one arrives at Theorem \ref{dani-discrete}. Before, let us fix our notation for inverses modulo a number.
\begin{definition}[Modular inverses]
Let $q$ be an integer different from zero. For any $a\in \Z$ coprime to $q$, we define $\overline{a}$ as the integer between $1$ and $q$ such that $\overline a a \equiv 1 \modu q$.
\end{definition}

\begin{lemma}[Fundamental period and discrete Dani]\label{period-made-explicit-discrete}

Let $N\ge 1$, $0<\delta<1/2$. Let $g_0\in G$ with coefficients $(R_0,\alpha_0,y_0^{-1})$ bounded by $\delta^{-1}$. Then the following statements are equivalent, unless $s<\delta^{-O(1)}N^{-1}$:

\begin{enumerate}

\item There exists $q\in\N$ and $\gamma\in \Gamma$ with the coefficients of $\gamma$ and $q$ bounded by $(\delta^{-1}\tau(q_2))^{O(1)}$  such that $g=\gamma g_0$ satisfies
\[
\|q \frac sR\|<(\delta^{-1}\tau(q_2))^{O(1)}N^{-1}, \, \hspace{20pt}
\|\left[q  \frac sR \right]\alpha\| < (\delta^{-1}\tau(q_2))^{O(1)}(sN^2)^{-1}.
\]
where $\tilde q_2$ and $q_2$ are the denominators in the expressions as  reduced fractions of $\frac{[[q\frac sR]\alpha]}{[q\frac sR]}$ and $\frac{[q\frac sR]}{q\tilde{q_2}^2}$ respectively.

\item There exists $y<1$ and an integer $q'$ with $y^{-\frac 12}/(\tau(q_2')\delta^{-1})^{O(1)} <q'<y^{-\frac 12}(\tau(q_2')\delta^{-1})^{O(1)}$ such that $\Gamma g_0=\Gamma(x+iy,\theta)$ with 
\[
\|q'x\|+N\|q'sy\|+(sN)^2 q' |\theta| y< y^{\frac 12 }(\delta^{-1}\tau(q_2'))^{O(1)},
\]
where $q_2'$ is the denominator in the expression as a reduced fraction of $\frac{[q' s y]}{q'}$.
\end{enumerate} 
\end{lemma}
\begin{remarks}
The proof actually gives $q_2'=q_2$. One can check that conditions in (i) assure that the coefficients of $\gamma$ and $q$ are always bounded by $(1+s)^{o(1)}\delta^{-O(1)}$.  One can see in the proof that from (i) we actually get a $y$ in (ii) satisfying $y\gg (1+s)^{-2-o(1)}$.
\end{remarks}
\begin{proof}
Let us begin by demonstrating that (ii) implies (i). Let us write explicitly the Iwasawa decomposition
\begin{equation}\label{IwasawaY}
g=h(x)a(y)k(\theta)=\begin{bmatrix}
-xy^{-\frac 12}\sin\theta +y^{\frac 12}\cos\theta & xy^{-\frac 12} \cos\theta +y^{\frac 12}\sin\theta   \\
 -y^{-\frac 12}\sin\theta & y^{-\frac 12} \cos\theta  \\
\end{bmatrix},
\end{equation}
or equivalently, with $W=\cot \theta$,
\begin{equation}\label{IwasawaR}
g=\begin{bmatrix}
-R^{-\frac 12}\alpha  &  R^{-\frac 12} \alpha W+R^{\frac 12}  \\
 -R^{-\frac 12} &R^{-\frac 12} W\\
\end{bmatrix}.
\end{equation}
From (ii) we have that
\[
x=\frac{a_1'}{q_1'}+y O(M) \hspace{20pt} q_1'\mid q', (a_1',q_1')=1, M=(\tau(q_2')\delta^{-1})^{O(1)}.
\]
On the other hand, suppose that $y^{-1}>4M^2 s^2$. Then, $s\sqrt y<1/2M$, so $q'sy=(q'\sqrt y)(s\sqrt y)<M(1/2M)<1/2$ which implies $\|q'sy\|=q'sy$, thus $q_2'=1$. But then, (ii) gives $M=\delta^{-O(1)}$ and $s<\delta^{-O(1)}/N$.

All this means that we can assume $y^{-1}<4M^2 s^2$. Then considering the matrix
\[
\gamma_{a_1'/q_1'}= \begin{bmatrix}
-\overline{a_1'} & d_1'  \\
q_1' & -a_1'    \\	
\end{bmatrix}
\] 
in $\Gamma$, we have
\[
g^*=\gamma_{a_1'/q_1'} g= \begin{bmatrix}
-\overline{a_1'} y^{\frac 12}(1+  (yq_1'\overline{a_1'})^{-1} \theta O(M) ) &  *   \\
q_1' y^{\frac 12} (1+\theta O(M))  &  q_1'y^{\frac 12} O(M) \\
\end{bmatrix}.
\]
Since  $|\theta|<M(sN)^{-2}$, this gives
\[
R^{-1}=q_1'^2 y (1+\frac{O(M)}{s^2N^2}), \hspace{20pt}
\alpha=-\frac{\overline{a_1'}}{q_1'}+\frac{1}{q_1'^2 y}\frac{O(M)}{s^2N^2},
\]
with $\alpha=\alpha(g^*)$, $R=R(g^*)$---the corresponding parameters associated to $g^*$. It also gives the inequalities $y(g^*)^{-1}\ll M q_1'^{2} y\ll M$, so if $q_1'^2 y<M^{-1}$ then $g^*$ is in the fundamental domain, but then since $y(g_0)<\delta^{-1}$ this implies $q_1'^2 y\gg M^{-1}$.  Now, 
\[
\frac{q'}{q_1'}\frac{s}{R}= q_1' (q'sy)+\frac{O(M)}{sN^2}
\]
so that $q=q'/q_1'<M$ satisfies
\[
\|q \frac{s}{R}\|<\frac{M}{N}, \hspace{20pt}  [q \frac{s}{R}]=q_1'[q' s y]
\]
and also, using the expression for $\alpha$, we have 
\[
[[q\frac{s}{R}]\alpha]=-\overline{a_1'}[q'sy], \hspace{20pt}
\|[q\frac{s}{R}]\alpha\|<\frac{M}{sN^2}.
\]
On the other hand
\[
\frac{[[q\frac s{R}]\alpha]}{[q\frac s{R}]}=\frac{-\overline{a_1'}}{q_1'}, \hspace{20pt}  \frac{[q\frac s{R}]}{q q_1'^2}=\frac{[q'sy]}{q'}
\]
so $\tilde{q}_2=q_1'$ and $q_2=q_2'$.  Finally, we have that the coefficients of the matrix $g^*$ are bounded by $M$, so since $\Gamma g^*=\Gamma g_0$ we have $g^*=\gamma g_0$, $\gamma$ with coefficients bounded by $M$ and (i) follows.

Now let us prove that (i) implies (ii). Let $s=s(g), R=R(g)$ and $W=W(g)$. We have
\begin{equation}\label{RConditions}
q\frac{s}{R}=[q\frac sR]+\frac{M}{N},   \hspace{20pt} \alpha=\frac{a_1}{q_1}+\frac{M}{s^2N^2}
\end{equation}
with $M=(\tau(q_2)\delta^{-1})^{O(1)}$ for some coprime integers $a_1, q_1$, with $q_1\mid [q\frac sR]$. We also have $|W|R^{-\frac 12}=y(g)^{-\frac 12}<M$. Let us consider $g_*=\gamma_{a_1/q_1}g=(x+iy, 0) k(\theta)$. Via \eqref{IwasawaR}
\begin{equation}\label{MGamma1}
g_*=\begin{bmatrix}
* &  -R^{\frac 12}\overline{a_1}(1 +\frac{M}{\overline{a_1}q_1 })  \\
\frac{q_1 M}{s^2 N^2} & R^{\frac 12} q_1 (1+\frac M{s^2N^2}) \\
\end{bmatrix},
\end{equation}
so considering the components in the lower row we have (by (\ref{IwasawaY})) 
\[
|\tan\theta|< \frac{M}{s^2N^2},    \hspace{20pt} y=\frac{1}{Rq_1^2}(1+\frac{M}{s^2 N^2})
\]
so taking $q'=qq_1$  we have
\[
q'sy=\frac{1}{q_1}q\frac{s}{R}(1+ \frac M{s^2N^2})=\frac{[q\frac sR]}{q_1}(1+\frac M{sN})=\frac{[q\frac sR]}{q_1}+\frac{M}{q_1 N}
\]
so $\|q'sy\|<My^{\frac 12}N^{-1}$. We also have $M^{-1}y^{-\frac 12}<q'<My^{-\frac 12}$ and $(sN)^2 q'y|\theta|<y^{\frac 12} M$. Now, by the second column in \eqref{MGamma1} and by \eqref{IwasawaY} we have
\[
x+y\tan\theta =\frac{-R^{-\frac 12}\overline{a_1}(1+\frac M{\overline{a_1}q_1})}{R^{\frac 12} q_1(1+\frac M{s^2N^2})}=-\frac{\overline{a_1}}{q_1}+ \frac M{q_1^2}
\]
so $q'x= -q\overline{a_1}+y^{\frac 12}M$ and then $\|q'x\|<My^{\frac 12}$.  Finally
\[
\frac{a_1}{q_1}=\frac{[[q\frac sR]\alpha]}{[q\frac sR]},  \hspace{20pt}   \frac{[q'sy]}{q'}=\frac{[q\frac sR]}{qq_1^2}
\]
so $q_1=\tilde{q}_2$ and $q_2'=q_2$.
\end{proof}

\section{Automorphic forms}\label{automorphic-forms}

A key to our analysis of the averages of functions along pieces of long periodic horocycles is the use of automorphic forms. We will need the sharpest known estimates for periods of the type
\begin{equation}
\int_0^1 f( 
\begin{bmatrix}\label{period-of-function}
1 & x \\
0 & 1 \\
\end{bmatrix}
g) e(-h x) \, dx 
\end{equation}
where $h\in \Z$, $g\in G$ and $f$ is a mildly varying function on $L_0^2 (\gammaq\backslash G)$ and $q\ge 1$ is an integer. Here the subzero indicates that $\int_{\gammaq\backslash G} f(g) \, dg = 0$ and $\gammaq$ denotes the standard Hecke congruence subgroup of $SL_2(\Z)$.

If $h=0$ and $g=a(y)$ then (\ref{period-of-function}) measures the equidistribution of the closed horocycle of period $1/y$ in $\Gamma_0 (q)\backslash G$. This can be studied using Eisenstein series and the precise rate of equidistribution is tied up with the Riemann Hypothesis (see \cite{sarnak-horocycle} for the case $q=1$, the rate is $O_{\epsilon}(y^{\frac 14+\epsilon})$ if and only if RH is true). For $h\neq 0$ the size of (\ref{period-of-function}) is controlled by the full spectral theory of $L_0^2(\Gamma_0 (q)\backslash G)$ and in particular the Ramanujan/Selberg conjectures for $GL_2/ \Q$ (see the appendix to \cite{sarnak-rankin}). In this case (\ref{period-of-function}) is closely related to the much studied shifted convolution problem. In order to bench-mark the upper bound that we are aiming for, consider the case that $f\in L_0^2(\Gamma_0(q)\backslash G)$ is $K$ invariant, that is $f(gk(\theta))=f(g)$. Thus $f=f(z)$ with $z=x+iy\in \Hy$ and (\ref{period-of-function}) is the period
\begin{equation}\label{period-poincare}
\int_0^1 f(x+iy) e(-hx) \, dx.
\end{equation}
We are assuming that $f$ is smooth and in $L_0^2(\Gamma_0(q)\backslash \Hy)$. We use the Sobolev norms
\begin{equation}
\|f\|_{W^{2d}}^2 = \int_{\gammaq\back \Hy} |f(z)|^2\,  d A(z) + \int_{\gammaq\back \Hy} |\Delta^d f(z)|^2 \, d A(z)
\end{equation}
where $d\ge 0$ is an integer, $dA=\frac{ dx \, dy}{y^2}$ is the area form and $\Delta$ the Laplacian for the hyperbolic metric. 

Expanding $f$ in the Laplacian spectrum of $\gammaq\back \Hy$ (see \cite{iw1}) with $\phi_j$ an orthonormal basis of cusp forms and $E_j (z,s)$ the corresponding Eisenstein series, $j=1,2,\ldots, \nu(q)$, yields 
\begin{equation}\label{expansion-eigenfunctions}
f(z)=\sum_{j\neq 0} \langle f,\phi_j \rangle \phi_j(z) + \sum_{j=1}^{\nu} \frac 1{2\pi} \int_{-\infty}^{\infty} \langle f, E_j(\cdot, \frac 12+it)\rangle E_j(z,\frac 12 +it) \, dt.
\end{equation}
Note that $\phi_0 (z) = 1 / \sqrt{\text{Vol}(\gammaq \back \Hy )}$ does not appear since we are assuming that
\begin{equation}\label{zero-mass-condition}
\int_{\gammaq \back \Hy }  f(z) \, dA(z) =0.
\end{equation}
Let $\lambda_j$ denote the Laplacian eigenvalue of $\phi_j$ and write $\lambda_j =\frac 14 + t_j^2$. $\phi_j$ may be expanded in a Fourier series (see \cite{iw1})
\begin{equation}\label{eigenfunction-fourier}
\phi_j(z)=\sum_{n\neq 0} \rho_j(n) y^{\frac 12} K_{it_j} (2\pi |n| y) e(nx).
\end{equation}
Using the Atkin-Lehner level raising operators one can choose the orthonormal basis $\phi_j$ to consist of new forms and old forms and then normalizing the coefficients in (\ref{expansion-eigenfunctions}) amounts to bounding the residues of $L(s,\phi_j \times \phi_j)$ at $s=1$. These are known to be bounded above and below by $(\lambda_j q)^{\pm\epsilon}$ respectively (see \cite{hoffstein-lockhart,iw2}). In this way one has the bound (see \cite{i-l-s} for details when $q$ is square free which essentially is the case of interest to us, and \cite{blomer-harcos,blomer-harcos-adeles} for the general $q$)
\begin{equation}\label{bound-coeff-eigenfunctions}
\rho_j(h) \ll_{\epsilon} (\lambda_j q h)^{\epsilon} q^{-\frac 12} \cosh (\frac{\pi t_j}2 ) h^{\theta}
\end{equation}
for any $\epsilon>0$, and where $\theta$ is an acceptable exponent for the Ramanujan/Selberg Conjecture. $\theta=7/64$ is known to be acceptable \cite{kim-sarnak} while $\theta=0$ is what is conjectured to be true. 

It follows from (\ref{bound-coeff-eigenfunctions}), (\ref{expansion-eigenfunctions}) and (\ref{eigenfunction-fourier}) that
\begin{equation}\label{period-bound} %\marginnote{\emph{I added the hyperbolic cosine}}
\int_0^1 f(x+iy) e(-hx)\, dx \ll_{\epsilon} \frac{y^{\frac 12} (qh)^{\epsilon} h^{\theta}}{\sqrt q} \sum_{j\neq 0} |\langle f, \phi_j \rangle | \, |K_{it_j} (2\pi |h| y) \lambda_j^{\epsilon}| 
\cosh (\frac{\pi t_j}2)
 \, + \text{cts}
\end{equation}
where the term ``cts'' is a similar contribution from the continuous spectrum and for which $\theta=0$ is known since the coefficients of Eisenstein series are unitary divisor sums.

The Bessel function $K$ satisfies the inequalities, say for $v\ll 1$ (see \cite{ba}) 
\begin{align}\label{bessel-bound} 
K_{\nu}(v)  & \ll_{\epsilon} v^{-\nu-\epsilon} & 0 \le \nu \le  1/2, \notag  \\
K_{it} (v) & \ll_{\epsilon}  v^{-\epsilon} & 0\le t\le 1, \\
e^{\frac{\pi}2 t} K_{it} (v)  & \ll_{\epsilon} v^{-\epsilon}  & 1 \le t <\infty. \notag
\end{align}
Hence from (\ref{bound-coeff-eigenfunctions}) we have that for $0<|hy| \ll 1$ 
\begin{align}
\int_0^1 f(x+it) e(-hx) \, dx  & \ll_{\epsilon} \frac{y^{\frac 12 -\theta-\epsilon} (qh)^{\epsilon}}{\sqrt q} \sum_{j\neq 0} |\langle f , \phi_j \rangle |\, \lambda_j^{\epsilon}\label{period-bound-2} \\
& \ll_{\epsilon} \frac{y^{\frac 12 -\theta-\epsilon} (qh)^{\epsilon}}{\sqrt q} (\sum_{j\neq 0} |\langle f,\phi_j\rangle |^2 \lambda_j^2 )^{\frac 12} \, ( \sum_{j\neq 0} \lambda_j^{-2+2\epsilon})^{\frac 12}.
\end{align}
Weyl's law for $\gammaq \back \Hy$ gives the uniform bound
\begin{equation}\label{weyl-law-peter}
\sum_{\lambda_j\le \lambda} 1 \ll \text{Vol}(\domainq) \lambda
\end{equation} 
for $\lambda \ge 1$ and $q\ge 1$. Hence
\begin{equation}\label{weyl-law-2-peter}
\sum_{j\neq 0} \lambda_j^{-2+2\epsilon}\ll \text{Vol}(\domainq) =q \prod_{p\mid q} (1+\frac 1p) \text{Vol} (\Gamma_0(1) \back \Hy) \ll q^{1+\epsilon}.
\end{equation}
We conclude that for $|h|y\ll 1$ and $\epsilon>0$ 
\begin{equation}\label{period-bound-final}
\int_0^1 f(x+iy) e(-hx) \, dx \ll_{\epsilon} y^{\frac 12 -\theta- \epsilon} q^{\epsilon} \|f\|_{W^2}.
\end{equation}
For $\theta=0$ (\ref{period-bound-final}) is sharp, that is it cannot be improved. To see this take for example $h=1$ in
\begin{equation}\label{period-spectral-expression}
\int_0^1 f(x+iy) e(-hx)\, dx = y^{\frac 12}\sum_{j\neq 0} \langle f, \phi_j \rangle \rho_j (h) K_{i t_j} (2\pi |h| y) +  \text{cts}.
\end{equation}
Choosing
\[
f(z)=\sum_{0\le t_j \le 1}  \overline{\rho_j (1)} \, \overline{K_{it_j} (2\pi y)} \phi_j(z)
\]
yields
\[ %\marginnote{\emph{I changed $=$ by $\asymp$ }}
\|f\|_{W^2}^2 \asymp\sum_{0\le t_j \le 1} |\rho_j(1)|^2 |K_{it_j} (2\pi y)|^2
\]
while  
\begin{equation}\label{period-special-function}%\marginnote{\emph{ I added $y^{\frac 12}$ and changed the equal sign by $\asymp$ }}
\int_0^1 f(x+iy) e(-h x) \, dx = y^{\frac 12} \sum_{0\le t_j \le 1} |\rho_j (1)|^2 |K_{it_j} (2\pi y) |^2  \asymp y^{\frac 12 } \|f\|_{W^2}^2.
\end{equation}
Recall that Iwaniec \cite{iw2} shows that
\begin{equation}\label{iwaniec-lower-bound} %\marginnote{\emph{I deleted the square in $\rho$}}
|\rho_j (1) |\gg_{\epsilon} (\lambda_j q)^{-\epsilon} \cosh (\frac{\pi t_j}2 ) /\sqrt q,
\end{equation}
hence changing $y$ a little if need be to make sure that $|K_{it_j}(2\pi y)| \gg 1$ for most $t_j\le 1$, we see that for this $f$
\begin{equation}\label{iwaniec-sobolev}
%\marginnote{\emph{isn't (3.16) enough to show that it is sharp?}}
\|f\|_{W^2} \gg_{\epsilon} (q\lambda_j)^{-\epsilon}.
\end{equation}
It follows from (\ref{iwaniec-sobolev}) and (\ref{period-special-function}) that (\ref{period-bound-final}) is sharp when $\theta=0$.

On the other hand if $\theta>0$ then (\ref{period-bound-final}) can be improved for $q$ in the range $y^{-\frac 12 +2\theta} \le q \le y^{-\frac 12}$. To do so we estimate the $j$-sum in (\ref{period-special-function}) using the Kuznetsov formula for $\domainq$, rather than invoking the sharpest bound (\ref{bound-coeff-eigenfunctions}) %\marginnote{\emph{Changed (3.17) by (3.7) } }
for the individual coefficients. 
One applies Kuznetsov with suitably chosen positive (on the spectral side) test functions and then using only Weil's upper bound for the Kloosterman sums that appear on the geometric side of the formula, we get (see \cite{iw2}); for $X\ge 1$, $h\neq 0$
\begin{equation}\label{density-exceptional}
\sum_{0<\lambda_j<\frac 14} |\rho_j(h)|^2 X^{4 |t_j|} \ll 1+ \frac{|h|^{\frac 12} X}q 
\end{equation}
and for $\lambda \ge 1$
\begin{equation}\label{density}
\sum_{\frac 14 \le \lambda_j \le \lambda } |\rho_j(h)|^2 \cosh(\pi t_j)\le  \lambda ( 1+ \frac{|h|^{\frac 12}}q ).
\end{equation}
Now for $0< |hy| \ll 1$, using (\ref{bessel-bound}) and (\ref{period-spectral-expression}) we have 
\begin{align}\label{period-density-bound}
\int_0^1 f(x+iy) e(hx)\, dx & \ll y^{\frac 12} (\sum_{0<\lambda_j <\frac 14} |\rho_j(h)|^2 |hy|^{-2|t_j|} )^{\frac 12} (\sum_{\lambda_j <\frac 14} |\langle f, \phi_j \rangle |^2 )^{\frac 12} \\
& + y^{\frac 12 }  (\sum_{\lambda_j \ge \frac 14} |\rho_j(h)|^2 |\lambda_j|^{-2} )^{\frac 12} (\sum_{\lambda_j \ge \frac 14} |\langle f, \phi_j \rangle |^2 \lambda_j^2 )^{\frac 12}. \notag
\end{align}
Taking $X=|hy|^{-\frac 12}$ in (\ref{density-exceptional}) and applying it to the first term on the right hand side of (\ref{period-density-bound}) and applying (\ref{density}) for the second term we arrive at
\begin{equation}\label{period-density-bound-final}
\int_0^1 f(x+iy) e(-hx)\, dx \ll y^{\frac 12} \|f\|_{W^2} (1+\frac{y^{-\frac 14}}{\sqrt q} + \frac{|h|^{\frac 14}}{\sqrt q}) \ll y^{\frac 12} \|f\|_{W^2} (1+\frac{y^{-\frac 14}}{\sqrt q})
\end{equation}
since $|h|\ll y^{-1}$.

We combine (\ref{period-bound-final}) with (\ref{period-density-bound-final}) to arrive at our strongest unconditional estimate; for $0<|hy|\ll 1$ and $\epsilon>0$,
\begin{equation}\label{period-bound-combined}
\int_0^1 f(x+iy) e(-hx)\, dx \ll_{\epsilon} (y^{-1} q)^{\epsilon} y^{\frac 12} \|f\|_{W^2} \min(y^{-\theta}, 1+\frac{y^{-\frac 14}}{\sqrt q}).
\end{equation}
For the application to the level of equidistribution in type I and II sums connected to sieving as we do in section \ref{density-hecke-orbit}, (\ref{period-bound-combined}) can be improved slightly when summing over $h$ in certain ranges; we leave that discussion for section \ref{density-hecke-orbit}.

In generalizing (\ref{period-bound-final}) and (\ref{period-bound-combined}) to functions on $\gammaq \back G$ it is natural to use the spectral decomposition of $L_0^2 (\gammaq \back G) $ into irreducibles under the action of $G$ by right translation on this space. This is the path chosen in \cite{blomer-harcos} and \cite{blomer-harcos-adeles} and we will follow these treatments closely modifying it as needed for our purposes. 

%%%%%%%%%%%%%%%%%%%%%%%%%%%%%%%%%%%%%%%%%%%%%%%

The Lie algebra $\mathfrak g$ of $G$ consists of the two by two matrices of trace zero. An element $X$ of $\mathfrak g$ gives rise to a left invariant differential operator on $C^{\infty}(G)$;
\begin{equation}
D_X f(g)= \frac{d}{dt} f (g \exp(tX))_{t=0}.
\end{equation} 
The operators $D_H$, $D_R$, $D_L$ with $H, R, L=R-V$ in (\ref{LieBasis}) generate the algebra of left invariant differential operators on $C^{\infty}(G)$. For $k\ge 0$ define the Sobolev $k$-norms on functions on $\gammaq\back G$ by
\begin{equation}
\|f\|_{W^k} := \sum_{\mathrm{ord}(D)\le k} \|Df\|_{L^2(\gammaq\back G)},
\end{equation}
where $D$ ranges over all monomials in $D_H$, $D_R$ and $D_L$ of degree at most $k$.  For a unitary representation $\pi$ of $G$, $\mathfrak g$ acts on the associated Hilbert space $V_{\pi}$ and one can define Sobolev norms of smooth vectors in the same way. The center of the algebra of differential operators is generated by the Casimir operator $\omega$ which in our basis is given by 
\begin{equation}
\omega = -\frac 14 ( D_H D_H + 2 D_R D_H + 2 D_H D_R ).
\end{equation}
In Iwasawa coordinates it is given by
\begin{equation}
\omega =-y^2 (\frac{\partial^2}{\partial x^2} + \frac{\partial^2}{\partial y^2} ) + y \frac{\partial^2}{\partial x \partial \theta}.
\end{equation}
We decompose $L_0^2 (\gammaq\back G)$ under right translation by $g\in G$ into irreducible subrepresentations $\pi$ of $G$. This extends the decomposition in 
(\ref{expansion-eigenfunctions}) to
\begin{equation}\label{spectral-decom}
L_0^2 (\gammaq\back G) = \int_{\widehat{G}} V_{\pi} \, d\mu (\pi),
\end{equation}
where $\pi$ ranges over $\widehat G$ the unitary dual of $G$ and $\mu$ is a measure on $\widehat G$ (it depends on $q$ of course) which corresponds to this spectral decomposition. It consists of a cuspidal part on which the spectrum is discrete (cuspidal $\Leftrightarrow \int_0^1 f( 
\begin{bmatrix}
1 & x \\
0 & 1 \\
\end{bmatrix}
g) \, dx =0, \forall g$) and a continuous part corresponding to an integral over unitary Eisenstein series. According to (\ref{spectral-decom}) for $f\in L_0^2 (\gammaq \back G)$ 
\begin{equation}\label{spectral-decom-func}
f=\int_{\widehat{G}} f_{\pi} \, d\mu (\pi)
\end{equation}
and
\begin{equation}\label{parseval}
\|f\|_{L_0^2 (\gammaq \back G)}^2 = \int_{\widehat{G}} \| f_{\pi} \|_{V_{\pi}}^2 \, d\mu (\pi).
\end{equation}
Since $\pi$ is irreducible and $\omega$ commutes with the $G$-action it follows that $\omega$ acts on smooth vectors in $V_{\pi}$ by a scalar which we denote by $\lambda_{\pi}$. Weyl's law for the principal and complementary series representations of $\widehat{G}$  together with the dimensions of the discrete series representation in $L_0^2 (\gammaq \back G)$ imply that for $T\ge 1$,
\begin{equation}\label{weyl-law}
\mu\{\pi: |\lambda_{\pi}|\le T\} \ll \mathrm{Vol} (\gammaq \back G)  (1+T).
\end{equation}
Note that for $a,b$ non-negative integers and $f\in V_{\pi}$ smooth we have from $\omega f =\lambda_{\pi} f$ that
\begin{equation}\label{spectral-coefficients-decay}
\|f\|_{W^a} \le (1+\lambda_{\pi})^{-b} \|f\|_{W^{a+b}}.
\end{equation}
Using the Hecke operators, Atkin-Lehner theory and choosing suitable bases to embed the space of new forms of a given level $t\mid q$ to level $q$ (see \cite{i-l-s,blomer-harcos-adeles}) we can further decompose $L_0^2 (\gammaq \back G )$ into an orthogonal direct sum/integral of irreducibles on which the Fourier coefficients satisfy the analogue of (\ref{bound-coeff-eigenfunctions}). In more detail if $\pi$ is an irreducible constituent in the above decomposition and $\pi$ is of level $q$ (by which we mean that $V_{\pi}$ has a vector which is a classical modular new form of level $q$) then the Whittaker functional
\begin{equation}
W_f(y):=\int_0^1 f(h(t)a(y)) e(-t) \, dt
\end{equation}
is non-zero as a function of $f$ a smooth vector in $V_{\pi}$. Moreover as is shown in \cite{blomer-harcos-adeles}  
\begin{equation}\label{relation-whittaker-norm}
\frac{\langle f, f \rangle}{\mathrm{Vol} (\gammaq \back G)}=c_{\pi} \langle W_f, W_f \rangle 
\end{equation}
where the second inner product is the standard inner product on $L^2(\R^*, \frac{dy}y)$ and in the Kirillov model for $\pi$ and $c_{\pi}$ satisfies \footnote{note our normalization of the Sobolev norms on $C^{\infty}(\gammaq\back G)$ and that of \cite{blomer-harcos-adeles} differ by a factor of $\mathrm{Vol}(\gammaq\back G)$}
\begin{equation}\label{bound-cpi}
(\lambda_{\pi} q)^{-\epsilon} \ll_{\epsilon} c_{\pi} \ll_{\epsilon} (\lambda_{\pi} q)^{\epsilon}.
\end{equation}
Moreover the $m$-th coefficient ($m\neq 0$) for our period integral satisfies the relation; for $f\in V_{\pi}$
\begin{equation}\label{fourier-coeffcient-whittaker}
\int_0^1 f(h(t) a(y) ) e(-m t) \, dt = \frac{\lambda_{\pi}(|m|)}{\sqrt{|m|}} W_f (|m| y).
\end{equation}
Here $\lambda_{\pi}(m)$ is the eigenvalue of the $m$-th Hecke operator on $V_{\pi}$. Recall that we are assuming that these satisfy
\begin{equation}\label{Peter}
\lambda_{\pi}(m) \ll \tau(|m|) |m|^{\theta}
\end{equation}
and similarly that the Laplace eigenvalue $\lambda_{\pi}=\frac 14+t_{\pi}^2$ (in the case that $\pi$ is spherical), if $it_{\pi}>0$ then 
\begin{equation}\label{bound-laplace-eigenvalue}
it_{\pi}\le \theta.
\end{equation}
The invariant differential operators on $V_{\pi}$ induce an action on $W_f$, namely
\[
D_X W_f = W_{D_X f}
\]
and using $D_H= 2y \frac d{dy}$, $D_R=2\pi i y$ and $D_L=\frac 1{2\pi i} (\frac{-\lambda_{\pi}}{y}+y\frac{d^2}{dy^2})$ we get using the various normalizations and (\ref{relation-whittaker-norm}), (\ref{bound-cpi}), (\ref{fourier-coeffcient-whittaker}), (\ref{Peter}) and (\ref{bound-laplace-eigenvalue}) (see \cite{blomer-harcos-adeles}) that for $b\ge 0$ fixed, $m\neq 0$ and $f\in V_{\pi}$
\begin{equation}\label{bound-W-improved}
\int_0^1 f(h(t) a(y)) e(-mt) \, dt \ll_{\epsilon, b} |\frac{\lambda_{\pi}q}{my}|^{\epsilon} \frac{1}{\sqrt q} \frac{\tau(|m|) y^{\frac 12-\theta}}{1+|my|^{b}} \|f\|_{W^{4+b}} .
\end{equation} 
As mentioned above, in \cite{blomer-harcos-adeles} it is shown how this analysis may be extended to any $\pi$ not just the ones of level $q$. The same bounds (\ref{bound-W-improved})
may be established for the Eisenstein spectrum and in that case one can take $\theta =0$.  Hence if $f\in L_0^2 (\gammaq\back G)$ is smooth we may apply (\ref{spectral-decom-func}), (\ref{parseval}) and (\ref{weyl-law}) to arrive at
\begin{align}\label{bound-fourier-coefficient}
\int_0^1 f(h(t)a(y))e(-mt) dt  & = \int_{\widehat G} \int_0^1 f_{\pi} (h(t) a(y)) e(-mt) \, dt \, d\mu(\pi) \notag \\
& \ll_{\epsilon, b}\frac{q^{\epsilon-\frac 12}}{|my|^{\epsilon}}  \frac{\tau(|m|) y^{\frac 12-\theta}}{1+|my|^{b}}\int_{\widehat G} \|f\|_{W^{4+b}} \lambda_{\pi}^{\epsilon} \, d\mu(\pi) \notag  \\
& \ll_{\epsilon, b}|\frac{q}{my}|^{\epsilon}  \frac{\tau(|m|) y^{\frac 12-\theta} }{1+|my|^{b}} \|f\|_{W^{6+b}}.
\end{align}
For $m=0$ only the Eisenstein series enter and a similar analysis of the constant term yields that for smooth $f$'s in $L_0^2 (\gammaq \back G)$ 
\begin{equation}\label{bound-zero-fourier-coef}
\int_0^1 f(h(t)a(y)) \, dt \ll q^{\epsilon} y^{\frac 12-\epsilon} \|f\|_{W^6}.
\end{equation}
(\ref{bound-fourier-coefficient}) and (\ref{bound-zero-fourier-coef}) are the generalizations of (\ref{period-bound-final}) to $\gammaq\back G$ and as we have noted these bounds are essentially contained in \cite{blomer-harcos-adeles}. The slight improvement (\ref{period-bound-combined}) can also be incorporated into this $\gammaq\back G$ analysis and this gives our main estimate for the period integral:
\begin{proposition}\label{period-bound-proposition}
For $\epsilon>0$, $b\ge 0$ fixed and $\theta$ admissible for the Ramanujan/Selberg conjecture for $GL_2/\Q$, if $f\in C^{\infty} (\gammaq\back G)$ and $m\in\Z$, then for $m\neq 0$
\[
\int_0^1 f(h(t)a(y)) e(-mt) \, dt \ll_{\epsilon, b} (\frac{q}{|m|y})^{\epsilon}  \frac{\tau(|m|) y^{\frac 12} \|f\|_{W^{6+b}}}{1+|my|^{b}} \min(y^{-\theta}, \frac{1}{y^{\epsilon}}+ \frac{y^{-\frac 14}}{y^{\epsilon} q^{\frac 12}})
\]
while for $m=0$, assuming $\int_X f \, d\mu_G =0$,
\[
\int_0^1 f(h(t) a(y) ) \, dt \ll_{\epsilon} (qy^{-1})^{\epsilon} y^{\frac 12} \|f\|_{W^6}.
\]
\end{proposition}
\begin{remark}
 For $\Gamma=\Gamma_0(1)$ the inequality (\ref{bound-laplace-eigenvalue}) is satisfied for $\theta=0$, and then in the bound for $m\neq 0$ in Proposition \ref{period-bound-proposition} we can substitute $y^{-\theta}$ by $|m|^{\theta}$  (see \cite{blomer-harcos}). For $\pi$ continuous we can substitute $y^{-\theta}$ by $\min(1,|t_{\pi}|)$ (see \cite{blomer-harcos,bruggeman-motohashi}).
\end{remark}

We can improve slightly Proposition \ref{period-bound-proposition} in average as follows. First, by Parseval we get
\begin{equation}\label{parseval-smooth}
 \sum_m |\tilde f (m,y)|^2 = \int_0^1 |f(h(x)a(y)|^2 \, dx \le \|f\|_{L^{\infty}}^2
\end{equation}
where $\tilde f(m,y)= \int_0^1 f(h(t)a(y)) e(-mt)\, dt$. Moreover, integrating by parts $b$ times and using the identity $h(\epsilon)a(y)=a(y)h(\epsilon/y)$ we get the bound $|\tilde f(m,y)|\ll_b (|m|y)^{-b} |\widetilde{D_R^b f}(m,y)|$, so by Cauchy's inequality and Parseval we have 
\begin{equation}\label{parseval-smooth-large}
\sum_{|m|>M} |\tilde f(m,y)|^2\ll_b(My)^{-b} \|D_R^b f\|_{L^{\infty}}^2. 
\end{equation}

To end this section we record some bounds for sums over Hecke eigenvalues $\lambda_{\pi}(n)$, that will be needed later. We restrict to $\Gamma=\Gamma_0(1)$ the full modular group as this is what will be used. Since such a  $\pi$ is either a fixed holomorphic form or a fixed Maass form it is known that (for the holomorphic case it is classical while for the Maass case see \cite{hafner}): for $T\ge 1$ and $x\in\R$ 
\begin{equation}\label{hafner-bound}
\sum_{m\le T} \lambda_{\pi} (m) e(mx) \ll_{\pi,\epsilon} T^{\frac 12+\epsilon}.
\end{equation}
We need control on the dependence in (\ref{hafner-bound}) of the implied constant. This can be done in terms of the eigenvalue $\lambda_{\pi}$ of $\pi$. It is also convenient for us here and elsewhere to work with the smooth normalized sums;
\begin{equation}\label{smooth-eigen-sum}
\sum_m \frac{\lambda(m)}{\sqrt m} e(mx) \psi(m u)
\end{equation}
for $u>0$ and small.

We can use the set up above with Whittaker functions to bound the sums in (\ref{smooth-eigen-sum}). For $f\in V_{\pi}$ one can control the $L^{\infty}$-norm of $f$ by its Sobolev norms (see \cite{blomer-harcos}):
\begin{equation}\label{PiDiscInfty}
\|f\|_{L^{\infty}}\ll \tilde\lambda_{\pi}^{3+o(1)}\|f\|_{W^3}
\end{equation}
for $\pi$ discrete, with $\tilde\lambda_{\pi}=\max(1,|\lambda_{\pi}|)$, and 
\begin{equation}
\label{PiContInfty}
\sup_{g=h(x)a(y)k(\theta)\in G}\frac{|f(g)|}{y^{1/2}+y^{-1/2}}\ll \tilde \lambda_{\pi}^{3+o(1)}\|f\|_{W^{3}}
\end{equation}
for $\pi$ continuous. 

Let $\pi$ be discrete. Suppose that $W_f$ has support in $(1,2)$ for some $f\in V_{\pi}$. From (\ref{relation-whittaker-norm})and the action of the differential operators on $W_{f}$ we deduce
\begin{equation}\label{bounding-f-with-W}
\|f\|_{W^b}\ll_b \tilde\lambda_{\pi}^{b+o(1)} \|W_{f}\|_{W^{2b}},
\end{equation}
with the norms for $W_f$ being the usual Sobolev norms for real functions.
On the other hand, for any $\psi\in C_0^{\infty}((1,2))$, there exists $f\in V_{\pi}$ with $W_{f}=\psi$. This comes from the fact that if $f\in V_{\pi}$ then $f_g(x)=f(xg) \in V_{\pi}$ for any $g\in G$, and $W_{f_{a(y)}}(u)=W_{f}(yu)$, $W_{f_{h(x)}}(u)=e(x)W_f(u)$.
Then,  due to \eqref{PiDiscInfty} and \eqref{bounding-f-with-W}, for any $0<u<1$ we have
\begin{equation}\label{bound-fourier-series-discrete}
\sum_m \frac{\lambda_{\pi}(m)} {\sqrt m}e(mx)\psi(mu)=f(h(x)a(u)) \ll \tilde\lambda_{\pi}^{6+o(1)}\|\psi\|_{W^6}.
\end{equation}
In the same way we have \footnote{Note that in terms of the dependence in $\lambda_{\pi}$ this is much weaker that Iwaniec's \cite{iw2}, $\sum_{m\le T}|\lambda_{\pi}(m)|^2\ll_{\epsilon} T\tilde \lambda_{\pi}^{\epsilon}$ and its extensions (\ref{fourth-moment}) and (\ref{eight-moment})}
\begin{equation}\label{parseval-discrete}
\left(\sum_m |\frac{\lambda_{\pi}(m)}{\sqrt m}|^2 |\psi(mu)|^2 \right)^{\frac 12}= \left(\int_0^1 |f(h(t)a(u))|^2 dt \right)^{\frac 12} \ll \tilde\lambda_{\pi}^{3+o(1)}\|\psi\|_{W^6}.
\end{equation}
For $\pi$ continuous, the suitable norm for the Hilbert space $V_{\pi}$ satisfies $\min(1,|t_{\pi}|)\|f\|_{V_{\pi}} \asymp  \|W_f\|_{L^2}$ and again from the action of the differential operators on $W_f$ we have
\begin{equation}\label{bounding-f-with-W-cont}
\min(1,|t_{\pi}|)\|f\|_{W^b}\ll_b  \, \tilde \lambda_{\pi}^{b+o(1)}  \|W_f\|_{W^{2b}}.
\end{equation}
Proceeding as in the discrete case, from \eqref{PiContInfty} and \eqref{bounding-f-with-W-cont} and sending $h(x)a(u)$ to the fundamental domain gives
\begin{equation}\label{WiltonCont}
\min(1,|t_{\pi}|)\sum_{m} \frac{\lambda_{\pi}(m)}{\sqrt m}e(mx)\psi(mu )\ll \, \min(\frac{u^{-1/2}}{q}, \frac{u^{1/2}}{\|qx\|})  \tilde\lambda_{\pi}^{6+o(1)}\|\psi\|_{W^6},
\end{equation}
where $q$ is the natural number smaller than $u^{-1/2}$ for which the quantity in the formula is the largest.  This last formula can be improved (by using Poisson Summation in $\R^2$, or by Voronoi formula) whenever $1/\|qx\|>1/uq$, in the sense of adding the decay factor
\[
(\frac{1/\|qx\|}{1/uq})^{-b}  
\]
for any fixed $b>0$ to the right of the inequality---but paying with a factor $O_b(\tilde\lambda_{\pi}^{O(b)}\|\psi\|_{W^{O(b)}})$.

\section{Effective equidistribution:  continuous algebraic measures}\label{effective-equid-non-discrete}

In this section we prove effective versions of Dani Theorem. We begin with some useful notation.  First we clarify the notion of continuous algebraic measure.

\begin{definition}[Continuous algebraic measure]
We shall say that an algebraic measure on $X$ is \emph{continuous} if it arises as limit when $T$ goes to infinity of the probability measures carried by the pieces of horocycle  $\Gamma g h(st)$, $t\in [0,T]$, for some fixed $g\in G, s\in \R$.
\end{definition}

 Since we are going to talk about probability measures, it is a good idea to normalize sums.

\begin{definition}[Expectation]
Let $J\subset \R$ be a finite set. Let $f:\R\to \mathbb C$. We define the \emph{expectation} of $f$ in $J$, and we write it as
\[
\mathbb E_{x\in J} f(x),
\]
as $\sum_{x\in J} f(x)/\sum_{x\in J} 1$. If $J$ is a bounded subinterval of $\R$, we define the expectation in the same way but using integrals.
\end{definition}

%\begin{definition}
%Let $J$ be as in the previous definition. Let $\psi:J\to X$. We define the \emph{measure supported on} $\psi(x)$, $x\in J$ as the one associating to any $f$ the expectation of $f\circ \psi$.
%\end{definition}

To deal with bounded functions that vary slowly, especially near the cusp, we use the following Lipschitz norms. 

\begin{definition}[Lipschitz norm]
Let $f:X\to \mathbb C$. We define the \emph{Lipschitz norm} of $f$ as
\[
\|f\|_{\mathrm{Lip}}=\|f\|_{L^{\infty}}+ \sup_{x\neq x'\in X} \frac{|f(x)-f(x')|}{\widehat{d_X} (x,x')}.
\]
with $\widehat{d_X}$ the metric on $X$ defined as
\[
\widehat{d_X}(x,x')=\min(d_X(x,x'), e^{-d_X(x,\Gamma I)}+e^{-d_X(x',\Gamma I)} ),
\]
where $I$ is the identity matrix in $G$.
\end{definition}
Notice that the use of the metric $\widehat{d_X}$ instead of $d_X$ in the previous definition implies that  $f$ can be seen as a function in the one-point compactification of $X$.

Finally, we define the quantitative concept of equidistribution that we shall use to describe our main results

%\begin{definition}[Distance of measures]
%Let $\mu_1$, $\mu_2$ two probability Borel measures on $X$. We define the \emph{distance} between them as
%\[
%d(\mu_1,\mu_2)= \sup_{\|f\|_{\text{Lip}}=1} |\int f \, d\mu_1 - \int f \, d\mu_2 |.
%\]
%\end{definition}

\begin{definition}[Effective equidistribution]
Let $0<\delta<1/2$. Let $g:J\to X$, with $J$ either a subinterval of $\mathbb Z$ or of $\mathbb R$. We say that $g(x), x\in J$ is $\delta$-equidistributed w.r.t. a Borel probability measure $\mu$ on $X$ if 
\[
|\E_{x\in J}  f(g(x))  -   \int f \, d\mu |  \le \delta \|f\|_{\text{Lip}}
\]
for any $f:X\to \mathbb C$.
\end{definition}
\begin{remark}
This concept of quantitative equidistribution is much weaker than the one used by Green and Tao in \cite{green-tao-equidistribution}. 
\end{remark}

The Lipschitz norm controls the Sobolev norms in the following sense: let $\delta>0$; for any $f\in C(X)$ with $\|f\|_{\mathrm{Lip}}=1$ there exists another function $f_{\delta}\in C(X)$ such that
\begin{equation}\label{lipschitz-controls-sobolev}
\|f-f_{\delta}\|_{L^{\infty}}<\delta
\end{equation}
and $\|Df_{\delta}\|_{L^{\infty}}\ll_{\mathrm{ord} D} \delta^{-\mathrm{ord} D}$ for any left-invariant differential operator. This will allow us from now on to substitute a function with bounded Lipschitz norm by one with ``bounded'' Sobolev norms\footnote{Note that in Sections \ref{effective-equid-non-discrete}, \ref{section-discrete-measures} and \ref{section-closure-primes} the entire discussion takes place on $X$; that is $\Gamma=SL_2(\Z)$. It is only in Section \ref{density-hecke-orbit} that the level $q$ of congruence subgroups is relevant.}. We can build $f_{\delta}$ as follows: for two functions $f,w\in L^2(SL_2(\R))$ define its convolution as
$f*w(g)=\int_{SL_2(\R)} f(gt) w(t^{-1}) \, d\mu(t);$
if $f$ is $\Gamma$-invariant then so is $f*w$.  
Moreover we can also write $f*w(g)=\int_{SL_2(\R)} f(t) w(t^{-1}g) \, d\mu(t)$ by the left invariance of $d\mu$. Pick $\psi\in C_0^{\infty}((0,1))$ a non-negative function and define $\psi_{\delta}(g)=c_{\delta} \psi(\delta^{-1}d_{SL_2(\R)}(g,I))$ with $c_{\delta}$ the constant that gives $\int_{SL_2(\R)} \psi_{\delta}(g^{-1})\, d\mu(g)=1$. Finally, defining $f_{\delta}=f*\psi_{\delta}$ one can check that it satisfies the desired properties.

Before beginning with our quantitative results, let us recall the following lemma on cancellation in oscillatory integrals.

\begin{lemma}[Integration by parts]\label{integration-parts}
Let $A>1$. Let $\eta \in C_0^{\infty}(1,2)$,  $F\in C^{\infty}(\R)$, $\eta$ with bounded derivatives and $|F'|\asymp A$, $F^{(j)}\ll_j A$ for any $j\in \N$.  Then
\[
\int \eta(t) e(F(t)) \, dt \ll A^{-\frac 1{|o(1)|}}.
\]
\end{lemma}
\begin{proof}
Write $F(t)=Af(t)$ and integrate by parts several times. Since the derivatives of $f$ are bounded, the result follows. 
\end{proof}

We are ready to give our first quantitative version of Dani's Theorem for a continuous orbit.

\begin{theorem}\label{implicit-dani-continuous}
Let $0<\delta<1/2$, $T>1$. For any $\xi\in X$ there exists a positive integer $j< \delta^{-O(1)}$  such $\xi h(t)$, $t$ in any subinterval of $[0,T]$ of length $T/j$, is $\delta$-equidistributed w.r.t. a  continuous algebraic measure in $X$. Moreover, it is the volume measure unless $y_T^{-1}< \delta^{-O(1)}$.
\end{theorem}
\begin{remark}
Notice that the use of the metric $\widehat{d_X}$ in the definition of $\delta$-equidistribution implies that Theorem \ref{implicit-dani-continuous} says nothing about the part of the horocycle that is near the cusp; the same can be said about any other result in sections 4 and 5.
\end{remark}
\begin{proof}
We can assume that $\delta^{-1}$ is not larger than a small power of $T$, because otherwise we could choose $j=T/\delta$ and then for any $t_0\in \R$, $\xi h(t)$, $t\in [t_0,t_0+\delta]$ would be $\delta$-equidistributed w.r.t. the algebraic measure carried by the point $\xi h(t_0)$.

If $y_T^{-1}>\delta^{-O(1)}$ we will prove the result with $j=1$; then, in that case we can assume that $|1\pm t/W_T|\gg \delta$ since this only deletes at most a proportion $\delta$ of the orbit. On the other hand, if for all $t$ in a certain interval $J$ we have $y_T/|1\pm t/W_T|^2\gg \delta^{-1}$ then by (\ref{parametrization-piece-flat}) we deduce that $\widehat d_X(\xi h(t),\xi h(t_0))\ll \delta $ for any fixed $t_0\in J$, and then $\xi h(t), t\in J$ will be $\delta$-equidistributed w.r.t. the algebraic measure carried by $\xi h(t_0)$; this implies that if $y_T^{-1}\le \delta^{-O(1)}$ we can assume that $1/|1\pm t/W_T|\ll \delta^{-O(1)}$. So in any case we can assume we are in the range $|1\pm t/W_T|^{-1}<\delta^{-O(1)}$.

Now, if the Theorem is false we can assume  that the Lipschitz norm of $f$ is 1 and that it has $\delta^{-O(1)}$-bounded Sobolev norms. Because of \eqref{parametrization-piece-flat}, there exists $y_T\le y^*< y_T\delta^{-O(1)}$ and $\eta\in C_0^{\infty}(\R)$ with $\|\eta\|_{W^j}<\delta^{-O(j)}$ and support an interval $I$ outside of $1/|1\pm t/W_T|<\delta^{-O(1)}$ such that
\[
|\mathbb E_{t\in I} \eta(\frac{t}{T})f_0(h(\alpha_T + \frac{y_T W_T}{1 \pm t/W_T}) a(y^*)) |\gg \delta,
\]
where $f_0$ equals $f-\int_0^1 f(h(x)a(y^*))\, dx$ in the case $y_T^{-1}\ll \delta^{-O(1)}$ and $f-\int_X f \, d\mu_G$ otherwise. Notice that this implies, by Proposition \ref{period-bound-proposition} in the second case, that $\int_0^1 f_0(h(t)a(y^*))\, dt\ll \delta^2 $.
But then, expanding in Fourier series $f_0(h(x)a(y))=\sum_m \tilde f_0(m,y) e(mx)$ we get
\[
\sum_{m\neq  0} |\tilde f_0(m,y^*)||\int \eta(t) e( \frac{my_TW_T}{1\pm t TW_T^{-1}})|\gg \delta.
\]
The derivative of the phase in the exponential is of size $|m|y_T T\gg |m|$, and then integrating by parts (Lemma \ref{integration-parts}) and using Proposition \ref{period-bound-proposition} we have for any large $b>0$ that
\[
\delta^{-O(1)}\sum_{m=1}^{\infty} y_T^{1/2} \tau(m)m^{\theta} (m y_T T)^{-b}\gg_b\delta.
\]
This implies that
$
\delta^{-O(1)}T^{-1/2}\gg \delta^{-O(1)} y_T^{1/2}(y_T T)^{-b}\gg_b \delta,
$
which is a contradiction.
\end{proof}

\emph{Remark:}  This result is related to Theorem 1 in \cite{strombergsson-integral} for the particular case $\Gamma=SL(2,\Z)$. It could also be proven by using quantitative mixing, as done in \cite{venkatesh} for $\Gamma\backslash G$ compact but taking care of the position of the orbit.

When we consider $\Gamma g$ fixed, we can make the result more explicit.
\begin{theorem}[Effective Dani, continuous orbit]\label{dani-continuous}
Let $\xi\in X$ fixed. The  continuous algebraic measure appearing in Theorem \ref{implicit-dani-continuous} is the volume measure unless there exists an integer $q< \delta^{-O(1)}$ such that 
\[
\|q\alpha\| < \delta^{-O(1)}T^{-1}.
\] 
\end{theorem}
\begin{proof}
This follows from Theorem \ref{implicit-dani-continuous}, Lemma \ref{period-made-explicit-continuous}  and the remark before it.
\end{proof}
\begin{remark}
Here and hereafter whenever we speak of ``algebraic measure'' it will mean one of the three algebraic measures appearing in Dani's Theorem (described in section \ref{horocycle-approximation}).
\end{remark}

As a corollary of this result we obtain Dani's Theorem for the continuous flow: if $\alpha$ is irrational, for any $0<\delta<1/2$ we can find $T=T_{\delta}$ large enough so that no $q$ exists as in the statement, and then $\Gamma g h(t)$, $t\in [0,T]$ is $\delta$-equidistributed w.r.t. the volume measure. 

Next we begin our study of the discrete orbit $\xi h(s)^n$, $n=0,1,\ldots N$. Our aim in the rest of this section is to understand, in terms of $\xi$ when any large piece of this orbit  is $\delta$-equidistributed w.r.t. a  continuous algebraic measure. We will be able to do it for not very large $s$; essentially in the range $s<(sN)^{1/5}$.

In the proofs we shall use the following version of the Stationary Phase principle (a more precise one can be found in \cite{hormander}):
\begin{lemma}[Stationary Phase]\label{stationary-phase}
Let $\eta\in C_0^{\infty}((1,2))$. Let $A>1$, $F\in C^{\infty}$ with $F''\asymp A$ and $F^{(j)}\ll_j A$ for $j\ge 3$ in the interval $(1/2,4)$. Then, we have
\[
\int \eta(t) e(F(t)) \, dt = \eta^*(t_F) e(F(t_F))A^{-1/2} + O(A^{-\frac 1{|o(1)|}})
\]
where $t_F$ is the only point where $F'$ vanishes, and $\eta^*\in C_0^{\infty}((1/2,4))$ depends just on $A$ and $F''$, and $\eta^*$ and its derivatives are bounded (in terms of $\eta$ only).
\end{lemma}
\begin{proof} 
First, let us write
\[
\int \eta(t) e(F(t)) \, dt =e(F(t_F))\int \eta(t_F+s) e(J(t_F,s)) \, ds
\]
where
\[
J(t,s)=F(t+s)-F(t)-F'(t)s=\int_0^s \int_0^u F''(t+v) \, dv \, du.
\]
Let us choose $C_1= A^{1/8}$. For some $\psi\in C_0^{\infty}((-2,-1)\cup(1,2))$ we can write
\[
1(s)=\sum_{j=-\infty}^{\infty} \psi(\frac{s}{2^j A^{-1/2}}) =\psi_0(\frac{s}{A^{-1/2}})+\psi_1(\frac{s}{C_1 A^{-1/2}}) + \sum_{1<2^j<C_1} \psi(\frac{s}{2^j A^{-1/2}}),
\]
with $\psi_0$ and $\psi_1$ coming from the sum over $2^j\le 1$ and $2^j\ge C_1$ respectively.
Decomposing the integral accordingly, one can easily show that the part corresponding to $\psi_1$ is bounded by $C_1^{-1/|o(1)|}<A^{-1/|o(1)|}$, and the rest can be expressed as
\[
A^{-1/2} e(F(t_F)) \eta^*(t_F)
\]
with
\[
\eta^*(t)= \eta_0^*(t)+\sum_{1<2^j<C_1} \eta_{2^j}(t),
\]
where
\[
\eta_C(t)= C\int \psi_C(t,s)e(C^2J_C(t,s))\, ds.
\]
with $\psi_C(t,s)=\psi(s)\eta(t+CA^{-1/2}s)$ and $J_C(t,s)=C^{-2}J(t,CA^{-1/2}s)$, and $\eta_0^*$ is defined as $\eta_{2^0}$ but changing $\psi$ by $\psi_0$.
The first thing to note is that the support of $\eta^*$ is contained in $(1/2,4)$ and the second is that $\eta^*$ just depends on $F''$, $A$ and $\eta$. It is easy to show that $\eta^*$ and its derivatives are bounded. For the rest, we have that the support of $\psi_C(t,\cdot)$ is contained in $(1,2)$; moreover
$
\frac{\partial}{\partial s} J_C \asymp 1,
$
and the partial derivatives of both $\psi_C$ and $J_C$ in $t$ and $s$ are uniformly bounded. Then, by integrating by parts several times, one shows that
$
(\frac{\partial}{\partial t})^j \eta_C (t) \ll_j C^{-1}
$
and then
$
(\frac{\partial}{\partial t})^j \eta^* (t) \ll_j 1
$
\end{proof}
In order to prove a quantitative version of Dani's Theorem for the discrete orbit $\Gamma gh(s)^n$, $n=0,1,\ldots N$, we split the analysis in three cases: when the piece of orbit is near to a closed horocycle ($\theta_T$ and $y_T^{-1}$ small, $T=sN$), when it is far from any closed horocycle ($\theta_T$ and $y_T^{-1}$ large), and the intermediate case.

We begin by the ``near'' case. Here, one can parametrize the orbit essentially as $\Gamma h(p(n))a(y_T)$ with $p(n)$ a quadratic polynomial. The following allows us to handle the distribution for such a sequence.

\begin{proposition}[Near to a closed horocycle]\label{near-to-closed-horo}
Let $0<\delta<1/2$. The sequence $s_n=\alpha+\beta n + \omega n^2+iy$, $n\le N$ is $\delta$-equidistributed with respect to the algebraic measure on the closed horocycle of period $y^{-1}$ unless either $\min(y^2/N|\omega|,Ny)<\delta^{-O(1)}$ or there exists a natural number $q$ smaller than $(\delta/\tau(q_2))^{-O(1)}(1+y^{-1/2})$ such that 
\begin{equation}\label{HeckeCondition}
\|q\alpha\|+N\|q\beta\|+N^2q|\omega|<y^{\frac 12}(\delta/\tau(q_2))^{-O(1)},
\end{equation}
where $q_2$ is the denominator in the expression as reduced fraction of $[q\beta]/q$.
\end{proposition}
\begin{proof} 
 The result is trivial for $y>\delta^{-1}$, since then for any $x_1,x_2\in \R$ we have $|f(x_2+iy)-f(x_1+iy)|\le \delta\|f\|_{\Lip}$. Otherwise, let us suppose that $s_n, n\le N$ does not satisfy the condition in the statement and that both $y^2/N|\omega|$ and $Ny$ are larger than $\delta^{-c}$ for $c>1$ a large constant. Then
\[
|\mathbb E_{n\le N} f(\alpha+\beta n+\omega n^2+iy)\eta(n/N)|\gg \delta,
\]
for some $f$ with  $\int_0^1 f(x+iy)\, dx=0$, $\|f\|_{W^j}\ll_j \delta^{-j}$ and $\eta\in C_0^{\infty}((0,1))$ with $\|\eta^{(j)}\|_{L^{\infty}}\ll_j \delta^{-j}$. 
Take $M=y^{-1}\delta^{-\sqrt c}$; there exist $q_{2}\le M$ and $a_{2}$ coprime to $q_{2}$ such that
\[
\beta=\frac {a_{2}}{q_{2}} +\epsilon  \hspace{20pt}  |\epsilon|\le\frac{1}{q_{2}M}.
\]
By splitting into arithmetic progressions modulo $q_{2}$ we have
\begin{equation}
\delta^{-O(1)}|\E _{b\le q_{2}} \, \E_{j\le N/q_2}  f(\alpha+\frac{a_{2}b}{q_{2}}+ \epsilon q_{2} j+\omega q_{2}^2 j^2+iy)\eta(\frac{j}{N/q_{2}})| > 1.
\end{equation}
Since $\|F^{(j)}\|_{L^{\infty}}\ll_j \delta^j$, with $F(t)=f(\alpha+a_{2}b/q_{2}+ \epsilon q_{2} t+\omega q_{2}^2 t^2+iy)\eta(\frac{t}{N/q_{2}})$,  we can see (for instance applying Poisson Summation and integrating by parts) that it is possible to substitute the inner sum by the integral
\begin{equation}\label{Beginning}
\delta^{-O(1)}|\E _{b\le q_{2}} \, \int f(\alpha+\frac{a_{2}b}{q_{2}}+ \epsilon Nt+\omega  N^2 t^2+iy)\eta(t) \, dt| > 1.
\end{equation}
At this point, let us remark that from the beginning we could assume the support of $\eta$ to be at distance $\delta$ from the point $-\epsilon(2\omega N)^{-1}$.
By the spectral decomposition \eqref{spectral-decom-func} and the Fourier series expansion  (taking into account that the zero coefficient vanishes) we have
\[
\delta^{-O(1)}|\sum_{m\neq 0} \frac{\lambda_{\pi}(m)}{\sqrt{|m|}} e(m\alpha)W_f(|m|y) \E_{b\le q_{2}} e(\frac{ma_{2}b}{q_{2}}) I(m)|> 1
\]
for some $\pi$  with $\lambda_{\pi}\ll \delta^{-O(1)}$ and $f\in V_{\pi}$ with $\|f\|_{W^j}\ll_j \delta^{-O(j)}$ and $I(m)= \int \eta(t) e(mN(\epsilon t+\omega N t^2) ) \, dt$. Then
\[
\delta^{-O(1)}|\sum_{k\neq 0} \frac{\lambda_{\pi}(q_{2}k)}{\sqrt{|q_{2}k|}}e(kq_{2}\alpha)W_f(|kq_{2}|y) I(kq_2)|> 1.
\]
By the multiplicativity of the Hecke eigenvalues  
\begin{equation}\label{multiplicativity}
\lambda_{\pi}(ab)=\sum_{d\mid (a,b)} \mu(d)\lambda_{\pi}(a/d)\lambda_{\pi}(b/d)
\end{equation}
we have
\[
\delta^{-O(1)}\frac{|\lambda_{\pi}(q_{2}/d)|}{\sqrt{dq_{2}}}|\sum_{m\neq 0} \frac{\lambda_{\pi}(m)}{\sqrt{ |m|}}e(mdq_{2}\alpha)W_f(|mdq_2| y) I(mdq_{2})|>\frac{1}{\tau(q_{2})} 
\]
for some $d\mid q_{2}$.
By splitting the integral defining $I(x)$ smoothly into $\delta^{-1}$ integrals and applying Lemma \ref{integration-parts}, taking into account our previous remark on the support of $\eta$, we have
\[
\delta^{-O(1)}\frac{|\lambda_{\pi}(q_{2}/d)|}{\sqrt{dq_2}}|\sum_{m\neq 0} \frac{\lambda_{\pi}(m)}{\sqrt{ |m|}}e(mdq_{2}\alpha)W_f(|mdq_{2}|y) \psi(mdq_2 u)|> \frac{1}{\tau(q_{2})} 
\]
with $u=\max(|\epsilon N|,|\omega| N^2)$ and $\psi$ a smooth function with
\[
\psi^{(j)}(x)\ll_{j,k} \delta^{-j} (1+|x|)^{-k} \hspace{20pt} k,j\ge 0.
\]
The function $\psi$ could depend on some of the parameters, but the bounds on the derivatives are absolute. 
Thus, if $\pi$ is discrete, splitting into dyadic intervals and using \eqref{bound-fourier-series-discrete}, Proposition \ref{period-bound-proposition} (with its remark) and \eqref{Peter} we have
\[
\delta^{-O(1)}(q_{2}/d)^{\theta} (dq_{2})^{-1/2}\min(1,yu^{-1})^{1/2} >  \tau(q_{2})^{-1} 
\] 
so that $q_{2}<\delta^{-O(1)}$ and $u< y\delta^{-O(1)}$. Now, there exists a natural number $q_1\le 1/\sqrt{\delta y}$ such that $\|q_1\alpha\|\ll\sqrt{\delta y}$; picking $q=q_1q_2$ one can check that the condition \eqref{HeckeCondition} is satisfied. If $\pi$ is continuous, again by Proposition \ref{period-bound-proposition} and \eqref{WiltonCont}---with the remarks after them---we have
\[
\delta^{-O(1)}  (dq_{2})^{-1/2}   \sqrt{dq_{2}y} \, \, \min (\frac{1}{q^* dq_{2}(y+u)},\frac 1{\|q^* dq_{2}\alpha\|}) > \tau(q_{2})^{-2} 
\]
for some $q^*\ll (dq_{2}(y+u))^{-1/2}$. Therefore for some $d^*
\mid q_{2}$ 
\[
q^*d^*q_{2}<(\delta/\tau(q_{2}))^{-O(1)}y^{-\frac 12}(1+uy^{-1})^{-1}
\]
and
$
\|q^*d^*q_{2}\alpha\|<(\delta/\tau(q_{2}))^{-O(1)}y^{1/2}.
$
These conditions imply \eqref{HeckeCondition}  with $q=q^* d^* q_2$.
\end{proof}

We now deal with the intermediate case. In the proof we simply use bounds for the Fourier coefficients.

\begin{proposition}[Intermediate case]\label{intermediate-case}
Let $0<\delta<1/2$, $N\ge 1$, $T=sN\ge 1$ and $g\in G$. Assume $s\delta^{-O(1)}<T^{1/4}$. The sequence $\Gamma g h(s)^n$, $n\le N$ is $\delta$-equidistributed w.r.t. the probability measure carried by $\Gamma g h(t)$, $t\in [0,T]$ unless either   both $y_T<\delta^{-O(1)}s/T$ and $|W_T| y_T<\delta^{-O(1)} s^2 T^{2\theta+o(1)}$ or both $y_T\delta^{-O(1)}> s^{-2} T^{-2\theta-o(1)}$ and $|W_T| y_T\delta^{-O(1)}> (T/s)^{2-o(1)} y_T^{2\theta}$. 
\end{proposition}
\begin{proof}
Let us assume for simplicity that $W_T>0$.
If $\Gamma g h(s)^n, n\le N$ is not $\delta$-equidistributed w.r.t. the measure carried by $\Gamma g h(t),t\in [0,T]$, splitting the orbit (\ref{parametrization-piece-flat}) and using (\ref{lipschitz-controls-sobolev})  we have that 
\[
|\mathbb E_{n/N \in I} f(h(x_{sn})a(y_*))\eta( n/N)|\gg\delta,
\]
for some $y_T<y_* <y_T \delta^{-O(1)}$, $f$ bounded with $\int_0^1 f(t+iy_*)\, dt=0$ and derivatives bounded by $\delta^{-O(1)}$, $\eta\in C_0^{\infty}$ with derivatives bounded by $\delta^{-O(1)}$ and supported on an interval $I$ of length larger than $\delta^2$, at distance at least $\delta^2$ from zero and such that $\delta^{-O(1)}|1\pm sn/W|>1$, with
\[
x_t=\alpha+yW(1\pm t/W)^{-1},
\]
$W=W_T$, $y=y_T$. Now, using  the Fourier expansion of $f(t+iy_*)$,  (\ref{parseval-smooth-large}) and Proposition \ref{period-bound-proposition} we have
\begin{equation}\label{previous}
\delta^{-O(1)}\sum_{\delta^{-O(1)} y^{-\epsilon}<|m| <\delta^{-O(1)}y^{-1}} |\tilde f(m,y_*)| |\mathbb E_{n} \eta(\frac nN) e(mx_{sn})|>1.
\end{equation}
Then,  we have (\ref{previous}) with the restriction that $|m|$ is either smaller or larger than  $\delta^{-O(1)} T^{\epsilon} W/T$. In the former possibility,  by Poisson summation in the inner sum, Lemmas \ref{integration-parts} and \ref{stationary-phase} and Proposition \ref{period-bound-proposition} we deduce that $y\delta^{-O(1)}> s^{-2} T^{-2\theta-o(1)}$ and $Wy\delta^{-O(1)}> (T/s)^{2-o(1)} y^{2\theta}$. In the latter, repeating the same steps we get that $Wy$ is at most $\delta^{-O(1)} s^2 T^{2\theta+o(1)}$, and using Cauchy's inequality and (\ref{parseval-smooth}) instead of Proposition \ref{period-bound-proposition} we have
\[
\delta^{-O(1)}\sum_{W/T<m<\delta^{-O(1)}y^{-1}} |\sum_{j\neq 0} b_j \eta_j^*(\frac{1-\sqrt{\frac{my}{j/s}}}{T/W}) \frac{e(2W\sqrt{myj/s})}{\sqrt{(j/s)T^2/W}}|^2 >1
\]
with $b_j=e(-Wj/s)$. Now, introducing a smooth function in the outer sum, expanding the square and changing the order of summation gives that either (from the diagonal terms) 
$\delta^{-O(1)}s T/yW> T^2 W^{-1}$
or 
\[
\delta^{-O(1)}|\sum_m  \rho(\frac{my/c-1}{a}) \psi( my/c)e(\epsilon cT\sqrt{my/c})|>Ts^{-2}(\epsilon c)^{-1+o(1)}
\]
for some $\psi, \rho \in C_0^{\infty}$ with support and derivatives bounded by $\delta^{-O(1)}$ and at distance $\delta^2$ from zero, $a<\delta^{-O(1)}$ and $\epsilon, c<\delta^{-O(1)}$, $\epsilon c\gg 1/s$.

The first possibility implies that
$y<\delta^{-O(1)}s/T$. The second possibility implies, writing $\rho(x)=\int \hat \rho(\theta) e(x\theta)\, d\theta$, that for some $\theta$ we have
\[
\delta^{-O(1)}|\sum_m \psi(my/c) e(m\theta) e(\epsilon c T\sqrt{my/c})|> Ts^{-2}(\epsilon c)^{-1+o(1)}
\]
which applying Poisson summation and Lemmas \ref{integration-parts} and \ref{stationary-phase} gives that $\delta^{-O(1)}s>T^{1/4}$, a contradiction.
\end{proof}

Now we treat the case in which the piece of horocycle is far from a closed one. The previous result worked for $y_T>s/T$, that is almost for the whole range. To handle the remaining range, we are going to split the piece of horocycle of length $T$ in pieces of length $\epsilon T$, $\epsilon <1$. Then we shall put those pieces in the fundamental domain $D_{X,\epsilon T}$, and then use the Fourier expansion (as in Proposition \ref{intermediate-case}). The advantage is that now we can get an extra cancellation from the fact that different pieces are essentially uncorrelated, and we can control this independence arithmetically. The ideas of the method come from the theory of exponential sums; it is essentially a modification of the one in \cite{jutila}.

\begin{proposition}[Far from a closed horocycle]\label{far-from-closed-horo}
Let $0<\delta<1/2$, $N\ge 1$, $T=sN\ge 1$ and $g\in G$. Assume $y_T<\delta^{-O(1)}s/T$ and $W_T y_T<\delta^{-O(1)}s^2T^{2\theta+o(1)}$. The sequence $\Gamma g h(s)^n$, $n\le N$ is $\delta$-equidistributed w.r.t. the volume measure on $X$ unless  $T^{1/5}s^{-1}$ is smaller than $\delta^{-O(1)}$. 
\end{proposition}
\begin{proof}
Let us assume for simplicity that $W_T>0$, and write $y=y_T$, $W=W_T$.
If $\Gamma g h(s)^n, n\le N$ is not $\delta$-equidistributed w.r.t. the volume measure on $X$, then by (\ref{parametrization-piece-flat}) and (\ref{lipschitz-controls-sobolev}) there exists $f\in L_0^2 (X)$ bounded and with derivatives bounded by $\delta^{-O(1)}$ and an interval $I$ of length satisfying $N|I|^{-1}<\delta^{-O(1)}$ such that
\[
|\mathbb E_{n\in I} f(h(x_{sn})a(y_*))| \gg\delta
\]
where $y\ll y_*<\delta^{-O(1)}y$, $\delta^{-O(1)}|1\pm sn/W|>1$ for any $n\in I$ and 
\[
x_t=\alpha+ yW(1\pm t/W)^{-1}.
\]
We can take $\alpha$ such that $|\alpha+yW|\le 1/2$.
Now we divide the sum into sums such that the argument in $h(\cdot)$ is near to a Farey fraction $a/q$ up to $q\le K$, with $K^2 y=(sT)^{-2/3}$. This choice for $K$ will make sense later; by now we can say that since $y<\delta^{-O(1)}s/T$, we have $K>\delta^{-O(1)}$ in the range $\delta^{-O(1)}s<T^{1/5}$. Precisely, we are going to take the interval around $a/q$ 
\[
(\frac aq - \frac 1{q(q+q')}, \frac aq +\frac 1{q(q+q'')})
\]
where $q'$ and $q''$ are the denominators of the Farey fractions (up to $K$) to the left and right of $a/q$ respectively. The point is that both $q+q'$ and $q+q''$ are comparable to $K$. In this way, we can write 
\[
1(x)= \sum_{a,q} \tilde\eta_{a, q}(\frac{x-a/q}{1/qK})
\]
where $a,q$ ranges over integers $a,q$ with $q\le K$, and $\tilde\eta_{a, q}$ are $C_0^{\infty}$ functions with $\|\tilde\eta_{a,q}^{(j)}\|_{L^{\infty}}\ll_j\delta^{-O(j)}$ with uniformly bounded support, and $\tilde\eta_{a,q}=0$ if $a,q$ are coprimes. In this way, we have
\[
|\mathbb E_{q\le K} \mathbb E_{a\in qL} \mathbb E_{n\in I_{a/q}}  f(h(x_{sn})a(y_*)) \tilde\eta_{a, q}(qK(x_{sn}-a/q))|\gg\delta
\]
with $L$ a subinterval of $[-\delta^{-O(1)}yT,\delta^{-O(1)}yT]$ with $yT/|L|<\delta^{-O(1)}$, and $I_{a/q}$ an interval of size $\asymp 1/syqK\gg \delta^{-O(1)}$ containing all $n$ such that $qK(x_{sn}-a/q)$ is in the support of $\tilde\eta_{a,q}$.  Taking into account that the fractions with $q<\delta^2 K$ give an amount $O(\delta^2)$, we have
\[
|\mathbb E_{\delta^2 K<q<K} \mathbb E_{a\in qL} \mathbb E_{n\in I_{a/q}}  f(h(x_{sn})a(y_*)) \eta^{\#}_{a,q}(q^2(x_{sn}-a/q))|\gg \delta
\]
with $\eta^{\#}_{a,q}(t)=\tilde\eta_{a,q}(tK/q)$. Now, splitting $\delta^2 K<q<K$ into $\delta^{-O(1)}$ intervals of equal length, we deduce that 
\[
\mathbb E_{|q-q_0|<\delta^c K} \mathbb E_{a\in qL} \mathbb E_{n\in I_{a/q}}  f(h(x_{sn})a(y_*)) \eta^{\#}_{a,q}(q^2(x_{sn}-a/q))|\gg \delta
\]
for some $q_0\in [\delta^2 K, K]$, $c>1$ a large constant. This implies that we can change $a\in qL$ by $a\in q_0 L$. Proceeding in the same way we have
\[
|\mathbb E_{|q-q_0|<\delta^c K} \mathbb E_{|a-a_0|<\delta^c K} \mathbb E_{n\in I_{a/q}}  f(h(x_{sn})a(y_*)) \eta_{a,q}(q^2(x_{sn}-a/q))|\gg \delta
\]
with $\supp \eta_{a,q}\subset [w_0,w_0+\delta^c]$ for some fixed $\delta^2\ll w_0\ll 1$ and $a_0\asymp \delta^{-O(1)}  yTK$, $\eta_{a,q}$ with the same properties as $\tilde\eta_{a,q}$. Defining $v_t=q^2(x_t-a/q)$ we can write
\[
|\mathbb E_q \mathbb E_a \mathbb E_n  f(h(a/q +v_{sn}q^{-2})a(y_*)) \eta_{a,q}(v_{sn})|\gg \delta
\]
with  and $a,q,n$ moving in the previous ranges. But multiplying by $\gamma_{a/q}\in\Gamma$ as in the proof of Lemma \ref{period-made-explicit-discrete} we have that $\gamma_{a/q} h(a/q +v/q^2)a(y_*)$ equals
\[
h(-\frac{\overline a}q-\frac{v}{v^2+(q^2 y_*)^2}) a(\frac{q^2y_*}{v^2+(q^2 y_*)^2})k(-\arccot\frac{v}{q^2 y_*})
\]
which, due to the restrictions on the ranges of $a,q, \eta_{a,q}$, is at distance $O(\delta^2)$ from
$
h(-\overline a/q-1/v+r_0) a(K^2y_0)
$
for some $y_0$ with $y<y_0< \delta^{-O(1)} y$ and $r_0$, both independent of $a, q$. Then
\[
|\mathbb E_{q} \mathbb E_a\mathbb E_{n}  f(h(-\frac{\overline a}q-\frac{1}{v_{sn}}+r_0) a(K^2y_0)) \eta_{a,q}(v_{sn})|\gg \delta
\]
The Fourier expansion $f(h(t)a(y))=\sum_m \tilde f(m,y)e(mt)$ gives that
\[
\delta^{-O(1)}|\sum_{m} \tilde f(m,K^2y_0) e(mr_0) \mathbb E_{a,q} e(-\frac{\overline a m}q)\mathbb E_n e(-\frac{m}{v_{sn}})\eta_{a,q}(v_{sn}) | >1.
\]
By Proposition \ref{period-bound-proposition} we can assume that $m$ is in the range $|m|>\delta^{-O(1)}T^{\epsilon}$ for some $\epsilon>0$. Thus, by Cauchy's inequality and (\ref{parseval-smooth-large}) we have
\[
\delta^{-O(1)}\mathbb E_{m} |\mathbb E_{a,q} e(-\frac{\overline a m}q)\mathbb E_n e(-\frac{m}{v_{sn}})\eta_{a,q}(v_{sn})|^2 > K^{2} y, 
\]
where $m$ moves in the range $\delta^{-O(1)}T^{\epsilon}<m<\delta^{-O(1)}(K^2 y)^{-1}$. By Poisson Summation in the $n$-sum, Lemmas \ref{integration-parts} and \ref{stationary-phase} and splitting of the obtained sum into intervals of the shape $[J,(1+\delta)J]$ we have
\[
\delta^{-O(1)}\mathbb E_{m} |\mathbb E_x\, b_x\, e(A\frac mF + B \sqrt{\frac mF})\eta_x^*(\frac{m}{F})|^2> c^{-1-o(1)}s^{-2}
\]
with $x=(a,q,j)$, $b_x$ independent of $m$ and bounded, $j$ in the range $cs<j<(1+\delta)cs$ for some $c<\delta^{-O(1)}$, $cs\ge 1$, and
\[
 F= \frac c{K^2 y}, \hspace{20pt} A=(-\frac{\overline a}q+\frac{q^{-2}}{a/q-\alpha})F,  \hspace{20pt} B= \frac{2Wy}{a/q-\alpha}\frac Kq \sqrt{\frac{j}{cs}} F. 
\]
We also can assume that we are summing just in the $a$'s for which $\overline a/q$ is contained in an interval of length $1/2$.
Introduction of a smooth factor in the $m$-sum and expansion of the square followed by Poisson Summation in $m$ and Lemmas \ref{integration-parts} and \ref{stationary-phase} gives that 
\begin{equation}\label{expectation_inequality}
 \delta^{-O(1)} \E_{(x,x')} \big(1_{x=x'}+ \frac{1_{\overline{a}/q\neq \overline{a'}/q'}}{\sqrt{|A-A'|}} + \frac{1_V}{\sqrt{1+|A-A'|}} \big) > \frac{c^{-o(1)}}{(cs)^{2}},
\end{equation}
where $A'=A(x')$, $B'=B(x')$ and $V$ is the subset of $(x,x')$ with $x\neq x'$ for which $\overline a/q= \overline{a'}/q'$ and either $|A-A'|$ and $|B-B'|$ are of comparable size or both of them are bounded by $\delta^{-O(1)}T^{\epsilon}$. This is so due to the restrictions in the ranges of $a,q,j$ and the fact that $a/q-\alpha\asymp Wy$.

The first summand in (\ref{expectation_inequality}) gives a contribution of $\delta^{-O(1)} (K yTK cs)^{-1}$ to the expectation, which equals $\delta^{-O(1)} c^{-1} (sT)^{-1/3}$. This is smaller than the right hand of (\ref{expectation_inequality}) in the range $\delta^{-O(1)}s<T^{1/5}$. In the second summand, again due to the restriction in the ranges of $a,q,j$, we have
\[
 |A-A'|\asymp |\overline a/q -\overline{a'}/q'| F.
\]
Then, by counting we see that the contribution of the terms with $|\overline a/q-\overline{a'}/q'|\asymp U$ is $\delta^{-O(1)}  U(UF)^{-1/2}$, so the full contribution of the second summand is $\delta^{-O(1)} F^{-1/2}$ which equals the one of the first summand---this is what motivated the election of $K$---and then is smaller than the right hand of (\ref{expectation_inequality}). 

Finally, we are going to show that the contribution of the third summand in (\ref{expectation_inequality}) is even smaller than the other two. For the terms in $V$ we have $q'=q$ and $a'=a+\lambda q$, with $0\neq \lambda \ll yT$, and then
\[
 |A-A'|\asymp \frac{|\lambda|}{q^2 (yW)^2} F, \hspace{20pt}    |B-B'|\asymp \frac Kq |\frac{\lambda}{a/q-\alpha}+1-\sqrt{j'/j}| F.
\]
Suppose first that $|1-\sqrt{j'/j}|\not\asymp |\lambda/(a/q-\alpha)|$. Then $|B-B'|$ is larger than $|\lambda/(a/q-\alpha)|$, so at least $\delta^{-O(1)}|A-A'|$. Moreover if $j'\neq j$ then $|B-B'|$ is at least $F/j$ which is larger than $\delta^{-O(1)}T^{\epsilon}$ in the range $s\delta^{-O(1)}<T^{1/2-\epsilon}$, so $(x,x')\not\in V$. Thus, we can assume $j'=j$, and then from (\ref{expectation_inequality}) we deduce that for $\lambda$ in some dyadic interval $\delta^{-O(1)}K^{-2}(|\lambda|/yT)(1/cs)$ is larger than  $T^{-o(1)}(cs)^{-2}$, so $\delta^{-O(1)}|\lambda| F>T^{1-o(1)}/s$; using this bound we have
\[
\delta^{-O(1)}|B-B'| \asymp \delta^{-O(1)}(Wy)^{-1}|\lambda|F\gg  (Wy)^{-1} T^{1-o(1)}/s
\]
which by $Wy<\delta^{-O(1)}s^2 T^{2\theta+o(1)}$ is larger than $\delta^{-O(1)} T^{\epsilon}$ in the range $\delta^{-O(1)}s<T^{1/5}$, since $\theta<1/5$. Therefore $(x,x')\not\in V$, a contradiction. 

Assume now that $|1-\sqrt{j'/j}|\asymp |\lambda/(a/q-\alpha)|$. Then
\begin{equation}\label{densityofj}
 1\le |j'-j|\asymp j|\lambda|(yW)^{-1}
\end{equation}
so the proportion of $j'$ is $O(|\lambda|/Wy)$. We can write $a=a_1+\mu q$ with $1\le a_1\le q$ and $\mu\ll yT$. For fixed $\lambda, a_1, q, j'$ and $j$, it is easy to see that $|A-A'|$ can be comparable to $|B-B'|$ for at most one $\mu$; in the same conditions, the number of $\mu$ for which both $|A-A'|$ and $|B-B'|$ can be at most $\delta^{-O(1)}T^{\epsilon}$ is $1+ \delta^{-O(1)}T^{\epsilon}(Wy)^2/(|\lambda| F)$. Applying the bound on $Wy$ from (\ref{densityofj}) we have
\[
  (Wy)^2 (|\lambda| F)^{-1} \ll |\lambda| j^2 c^{-1}(sT)^{-2/3}\ll yT s^2 (sT)^{-2/3},
\]
which, due to the bound $y\ll \delta^{-O(1)} s/T$, is smaller than $1$ in the range $s\delta^{-O(1)}<T^{2/7}$. Then, the proportion of $\mu$ in $V$ is $\delta^{-O(1)}T^{\epsilon}/yT$, so the contribution from the third summand is the supremum in $\lambda\ll yT$ of 
\[
 \delta^{-O(1)}T^{\epsilon} K^{-2}\frac{|\lambda|}{Wy}\frac{1}{yT} (\frac{|\lambda|}{K^2 (yW)^2} F)^{-1/2}
\]
which is at most $\delta^{-O(1)} T^{\epsilon}c^{-1/2} T^{-1/2}$ and then smaller than then right hand of (\ref{expectation_inequality}) in the range $\delta^{-O(1)} s<T^{1/4-\epsilon}$.

\end{proof}

\begin{remarks} (a) It would be possible to improve the range for $s$ a little by treating non-trivially the exponential sums appearing in the proof in the last application of Stationary Phase. This would improve the range also in Theorem \ref{dani-discrete}.

(b) We could prove this result also (with a smaller range for $s$) by proceeding as in the proof of Proposition \ref{intermediate-case}, but treating the sums $\sum_n \lambda_{\pi} (n) e(F(n))$ with van der Corput's Lemma and shifted convolution. One could also prove it in a quicker way (again with $s$ smaller) by relating the discrete orbit to the continuous one as in Lemma 3.1 of \cite{venkatesh}, and then using Theorem \ref{implicit-dani-continuous}. 
\end{remarks}

Now we join the previous cases to prove an effective version of Dani's result for sums

\begin{theorem}[Effective Dani, discrete orbit]\label{dani-discrete}
Let $\xi \in X$ fixed. Let $0<\delta <1/2$, $N\ge 1$ and $s>0$. There exists a positive integer $j<\delta^{-O(1)}$ such that the sequence $\xi h(s)^n$, with $n$ in any subinterval of $n\le N$ of length $N/j$, is $\delta$-equidistributed w.r.t. a  continuous algebraic measure on $X$ unless either $(sN)^{1/5}s^{-1}<\delta^{-O(1)}$ or the conditions in Lemma \ref{period-made-explicit-discrete} are satisfied. Moreover both $j$ and the measure are the ones in Theorem \ref{implicit-dani-continuous}---the continuous case.
\end{theorem}
\begin{proof} Let $\xi = \Gamma g$.  As before, the case $\delta^{-O(1)}>T=sN$ is trivial. Otherwise, notice that if we show that $\xi h(s)^n, n\le N$ is $\delta$-equidistributed w.r.t the probability measure carried by $\xi h(t), t\in [0,T]$, by Theorem \ref{implicit-dani-continuous} we are done. Thus, by Propositions \ref{far-from-closed-horo} and \ref{intermediate-case} we get the result unless both $y_T\delta^{-O(1)}>s^{-2}T^{-2\theta-o(1)}$ and $|W_{T} y_T|\delta^{-O(1)} >  (T/s)^{2-o(1)} y_T^{2\theta}$, with $T=sN$. In this case, we have that
\[
y_T>\delta^{-O(1)}/N, \hspace{30pt} |W_T y_T|> \delta^{-O(1)} (y_T T^{3/2}+sT),
\]
in the range $\delta^{-O(1)} s<T^{1/5}$, since $\theta<1/5$, so
\[
\Gamma gh(s)^n=h(x+y_T sn+y_T\frac{s^2n^2}{W_T})a(y_T)+O(\delta^2) 
\]
and we can apply Proposition \ref{near-to-closed-horo}, which gives that $\xi h(s)^n, n\le N$ is $\delta$-equidistributed w.r.t. the algebraic measure on the closed horocycle of period $y_T^{-1}$
unless there exists an integer $q< (\delta/\tau(q_2))^{-O(1)}y_T^{-1/2}$ such that
\[
\|qx\|+N\|qsy_T\|+(sN)^2 q y_T/|W_T| < y_T^{\frac 12} (\delta/\tau(q_2))^{-O(1)}.
\]
Actually we must have $q>(\delta/\tau(q_2))^{O(1)}y_T^{-1/2}$ because $\xi$ is fixed. But these are the conditions in Lemma \ref{period-made-explicit-discrete}.
\end{proof}
\begin{remark}
 The proof actually shows that $\xi h(s)^n, n\le N$ is $\delta$-equidistributed w.r.t the measure carried by $\xi h(t), t\in[0,sN]$ unless either $(sN)^{1/5}s^{-1}<\delta^{-O(1)}$ or the conditions in Lemma \ref{period-made-explicit-discrete} are satisfied.
\end{remark}

As a corollary of this result we obtain Dani's Theorem for discrete orbits.

\begin{corollary}[Dani Theorem]\label{dani}
Let $\xi\in X$ with $\alpha$ irrational. Then $\xi h(s)^n$, $n\le N$ becomes equidistributed w.r.t the volume measure on $X$ as $N$ goes to infinity.
\end{corollary}

\section{Effective equidistribution:  non-continuous algebraic measures}\label{section-discrete-measures}

In the previous section we gave necessary conditions on $\xi$ for the probability measure carried by any large piece of the orbit $\xi h(s)^n$, $n\le N$ to be near to some continuous algebraic measure. This is all we need in order to prove our results for primes. However, it remains to answer the following questions: is it possible for that measure to be always near to an---either  continuous or not---algebraic measure?  Are the conditions on $\xi$ really sharp for the measure to be near to a continuous algebraic measure?  In this section we will show that the answer to both questions is essentially ``yes'', as long as $(sN)^{1/5}s^{-1}>\delta^{-O(1)}$.

Theorem \ref{dani-discrete} says that, for $(sN)^{1/5}s^{-1}>\delta^{-O(1)}$, if the sequence $\xi h(s)^n$, $n\le N$  is not equidistributed w.r.t. a  continuous algebraic measure, then it is near to a sequence
\begin{equation}\label{sn-definition}
s_n=\frac{A+Bn}{q}+ \epsilon_4   \frac{\epsilon_1+\epsilon_2\frac nN +\epsilon_3 (\frac nN)^2+i\epsilon_4}{q^2} 
\end{equation}
for some integers $A, B, q$ with $(A, B, q)=1$ and $\epsilon_j \ll (\tau(q_2)\delta^{-1})^{O(1)}$, with $\frac{B}{q}=\frac{a_2}{q_2}$, $(a_2,q_2)=1$. 
We also have $\epsilon_4/q^2\gg (1+s)^{-2-o(1)}$ from the remarks after Lemma \ref{period-made-explicit-discrete}.
A key ingredient to understand $s_n$ is going to be to understand its restriction to an interval $L< n \le L+q$; there it is very near to a sequence
\begin{equation}\label{rationals-1}
\frac{A+Bn}q+ M_2(\frac{M_1}{q^2}+i\frac{M_2}{q^2})    \hspace{20pt}   \text{with } n \text{ mod } q_2,
\end{equation}
where $M_1, M_2 \ll \tau(q_2)^{O(1)}$. We can describe the points 
\[
\frac{A+Bn}q          \hspace{20pt}    n\text{ mod } q_2
\]
$\text{ mod } 1$ in a more convenient way. First, we can write 
\[
\frac{A+q_1 a_2 n}{q_1q_2}          \hspace{20pt}    n\text{ mod } q_2
\]
with $q_1=q/q_2$ and then $(A, q_1)=1$. Since $a_2$ is coprime to $q_2$ we can write
\[
\frac{A+q_1 n}{q_1q_2}          \hspace{20pt}    n\text{ mod } q_2.
\]
Next, we can write $q_2=q_2' q_2''$ with $q_2'$ coprime to $q_1$ and $q_1$ a multiple of every prime dividing $q_2''$. Since $(q_1,q_2')=1$, the equation
\[
A +q_1 x \equiv 0 \text{ mod } q_2'
\]
has solution in $x$, and then we can write
 \[
\frac{q_2' h + q_1 n}{q_1q_2} \hspace{20pt}   \text{with } n \text{ mod } q_2
\]
with $(h,q_1)=1$ and then $(h,q_1q_2'')=1$. Finally, writing $m=vq_2'+uq_2''$ with $u,v$ integers the points in \eqref{rationals-1} can be described as
\[
\frac u{q_2'} +\frac v{q_2''} + \frac{h}{q_1q_2''}+M_2(\frac{M_1}{q^2}+i\frac{M_2}{q^2})  \hspace{20pt}    u \text{ mod }q_2' ,\, \, v \text{ mod } q_2''.
\]
To further study the behaviour of these points in $X$ it is natural to split them into classes $C_l$ for any $l$ a divisor of $q_2'$, each $C_l$ corresponding to the points such that $(u, q_2')=l$. 

For any point $g$ of $SL(2,\mathbb R)$ of the form;
\[
g=(\frac{b}{q}+M_2 \frac{M_1+iM_2}{q^2},0)
\]
with $(b,q)=1$, we can multiply to the left by $\gamma_{b/q}\in SL(2,\mathbb Z)$, with
\[
\gamma_{b/q} = 
\begin{bmatrix}
-\overline b &  *  \\
q    &     -b
\end{bmatrix} 
\]
obtaining
\begin{equation}\label{up}
g^*=\gamma_{b/q} g = (-\frac{\overline b}q -\frac{M_1/M_2}{M_1^2+M_2^2} + \frac{i}{M_1^2+M_2^2},  -2\arccot  \frac{M_1}{M_2}).
\end{equation}
So in our setting it is easy to see that $g\mapsto g^*$ sends the points in $C_1$ to
\[
(\frac{u}{q_2'}+\frac{\overline{q_2'}^2 \overline{q_1 v+h}}{q_1q_2''} -\frac{M_1/M_2}{M_1^2+M_2^2} + \frac{i}{M_1^2+M_2^2},  -2\arccot  \frac{M_1}{M_2}).
\]
where $\overline{x}$ is the inverse modulo $q_1q_2''$.  One can rewrite these points as
\[
(\frac{u}{q_2'}+\frac{\overline{q_2'}^2 (q_1 v+\tilde h)}{q_1q_2''} -\frac{M_1/M_2}{M_1^2+M_2^2} + \frac{i}{M_1^2+M_2^2},  -2 \arccot  \frac{M_1}{M_2}).
\]
with $\tilde h$ an inverse of $h$ modulo $q_1$. But then this is the same as
\[
(\frac n{q_2}+\frac{\overline{q_2'}^2 \tilde h}{q_1q_2''} -\frac{M_1/M_2}{M_1^2+M_2^2} + \frac{i}{M_1^2+M_2^2}, -2\arccot  \frac{M_1}{M_2})
\]
with $n \modu q_2$, $(n,q_2')=1$.
In general, for any $l\mid q_2'$, everything works the same way but dividing $q_2', M_1$ and $M_2$ by $l$, and then the points of $C_l$ can be seen as
\begin{equation}\label{lines}
(\frac{n}{q_2/l}+l^2 z,  -2\arccot  \frac{M_1}{M_2}),
\end{equation}
with $n \modu q_2/l$, $(n,q_2'/l)=1$ and
\[
z=\frac{\overline{q_2'}^2 \tilde h}{q_1q_2''} -\frac{M_1/M_2}{M_1^2+M_2^2} + \frac{i}{M_1^2+M_2^2}.
\]

We are finally prepared to state the main result concerning the near to a closed horocycle case

\begin{proposition}\label{discrete-measures}
Let $0<\delta<1/2$. There exists $j< \delta^{-O(1)}$ such that when $n$ is restricted to any subinterval of $n\le N$ of length $N/j$ the sequence $s_n$ is  $\delta$-equidistributed w.r.t. some algebraic measure in $X$. 
\end{proposition}
\begin{proof}
 If $|\epsilon_2|+|\epsilon_3|<\delta^{-O(1)} \epsilon_4$ then $s_n$ and $s_m$ are at distance $O(\delta^{-O(1)} |n-m|/N)$, hence the result follows---each algebraic measure is 
supported on a point. Thus, from now on we assume $|\epsilon_2|+|\epsilon_3|\gg \delta^{-O(1)} \epsilon_4$. 
 
  Let us treat first the case $q_2\ll \delta^{-O(1)}$; if $n$ is in an interval $J$ of length $\delta^{c}N$ containing the point $N'$ (we should choose $N'$ such that $t=N'/N$ makes $|\epsilon_1+\epsilon_2 t+\epsilon_3 t^2|\gg|\epsilon_1|+|\epsilon_2|+|\epsilon_3|$), we select the sequence
\[
r_n=\frac{A+Bn}{pq} + \epsilon_4   \frac{\epsilon_1+\epsilon_2\frac {N'}N +\epsilon_3 (\frac {N'}N)^2+i\epsilon_4}{(pq)^2}      \hspace{30pt}  n \text{ mod } pq_2
\]
for any prime $p> \delta^{-O(1)}$, which carries an algebraic measure $\mu$.  Let us see that $s_n$ in the interval $J$ is $\delta$-equidistributed w.r.t. $\mu$.  For that, let us split $J$ into arithmetic progressions $n\equiv n_0 \modu q_2$; for each one we have (due to (\ref{up}))
\[
\Gamma s_n=\Gamma (x_0-\frac{b_n}{\epsilon_4}\frac{q_0^2}{b_n^2+\epsilon_4^2}+\frac{i q_0^2}{b_n^2+\epsilon_4^2}, - 2\arccot \frac{b_n}{\epsilon_4} )
\]
for $b_n=\epsilon_1+\epsilon_2\frac nN +\epsilon_3 (\frac nN)^2$ and $x_0$ and $q_0\mid q$ depending on $n_0$. Then
\[
\Gamma s_n = x_0-\frac{\epsilon_4^{-1}q_0^2 }{b_n^2}+\frac{i q_0^2}{b_{N'}^2}+O(\delta)
\]
for most $n$'s in the arithmetic progression contained in $J$. Since $\delta^{-O(1)} < |\epsilon_4^{-1}|<\delta^{-O(1)} s^{2+o(1)}<N^{O(1)}$ by classical methods in exponential sums one can prove (see for instance \cite{iwaniec-kowalski})
\[
\E_{n\in J, n\equiv n_0 \modu{q_2}} e(k \frac{\epsilon_4^{-1}q_0^2 }{b_n^2} ) \ll \delta^{c'}
\]
for large $c'$ and any $k<\delta^{-O(1)}$. Then, we have that $s_n$, $n\in J$, $n\equiv n_0 \modu q_2$ is $\delta$-equidistributed w.r.t. the measure carried by
\[
t+i\frac{q_0^2}{b_{N'}^2}      \hspace{20pt}    t\in [0,1].
\]
On the other hand we have, again by (\ref{up}), that $r_n$
with $n\in J$, $n\equiv n_0 \modu q_2$ is $\delta$-equidistributed w.r.t. the measure carried by
\[
(x_0'+\frac{m}{p}-\frac{b_{N'}}{\epsilon_4}\frac{q_0^2}{b_{N'}^2+\epsilon_4^2}+i\frac{q_0^2}{b_{N'}^2+\epsilon_4^2}, - 2\arccot \frac{b_{N'}}{\epsilon_4})   \hspace{20pt} m  \modu p
\]
for some $x_0'$ depending on $n_0$. Since any two consecutive points of this sequence are at distance $p^{-1}\delta^{-O(1)}<\delta^2$, we have that $r_n$ $n\in J$, $n\equiv n_0 \modu q_2$ is $\delta$-equidistributed w.r. t. the measure carried by
\[
t+i\frac{q_0^2}{b_{N'}^2}      \hspace{20pt}    t\in [0,1],
\]
so $s_n$, $n\in J$ is $\delta$-equidistributed w.r.t. $\mu$. 

It remains the case $|\epsilon_2|+|\epsilon_3|> \delta^{-O(1)} \epsilon_4$ and $q_2> \delta^{-O(1)}$. In this case we can take the algebraic measure $\mu$ carried by
\[
r_n=\frac{A+Bn}q+ \epsilon_4   \frac{\epsilon_1+\epsilon_2\frac {N'}N +\epsilon_3 (\frac {N'}N)^2+i\epsilon_4}{q^2}      \hspace{30pt}  n \text{ mod } q_2,
\]
to approximate $s_n$, $n\in J$. To see that $s_n$, $n\in J$ is $\delta$-equidistributed w.r.t. $\mu$, we shall show the same restricting $n$ to any subinterval of $J$ of length $q_2$; since $q<\delta^{-O(1)}y^{-\frac 12}$ and $y^{-1}\ll sN$ we have that $q<\delta N$; then $s_n$ restricted to a subinterval of $J$ of length $q_2$ satisfies
\[
s_n= \frac{A+Bn}q+ \epsilon_4   \frac{\epsilon_1+\epsilon_2\frac {N''}N +\epsilon_3 (\frac {N''}N)^2+i\epsilon_4}{q^2} +O(\delta)      \hspace{30pt}  n \text{ mod } q_2,
\]
with $N''\in J$. Now we restrict $n$ further to the set $C_l$, for some $l\mid q_2'$, of $n$'s satisfying $(n,q_2')=l$; it is necessary to look just to the $l$'s with $l<\tau(q_2')\delta^{-O(1)}$, because the others give a negligible contribution. Due to (\ref{lines}) the restriction of $r_n$ to $C_l$ is
\begin{equation}\label{line}
(\frac{n}{q_2/l}+l^2 x_1-\frac{b_{N'}}{\epsilon_4}\frac{l^2}{b_{N'}^2+\epsilon_4^2} + i \frac{l^2}{b_{N'}^2+\epsilon_4^2},  -2\arccot  \frac{b_{N'}}{\epsilon_4}).
\end{equation}
with $n \modu q_2/l$, $(n,q_2'/l)=1$, and $x_1=\frac{\overline{q_2'}^2 \tilde h}{q_1q_2''}$. 
Now, it is easy to show that for any $x, \epsilon>0$ and $j, d \in \N$ we have
\[
\sum_{x<n\le x+\epsilon jd, (n,d)=1} 1 = \epsilon \sum_{n\le j d, (n,d)=1} 1   +  O(\tau(d)). 
\]
Applying it for $d=q_2'/l$, $j=q_2/q_2'$ and $\epsilon=\delta l^2/(b_{N'}^2+\epsilon_4^2)$, since  
$\epsilon j \phi(d)\gg \delta^{-O(1)}\tau(d)$, we have
that the sequence in (\ref{line}) is $\delta$-equidistributed w.r.t. the measure carried by
\[
(t+i \frac{l^2}{b_{N'}^2+\epsilon_4^2}, -2\arccot  \frac{b_{N'}}{\epsilon_4})    \hspace{20pt}  t\in [0,1].
\]
One can proceed in the same way for the restriction of $s_n$, obtaining that it is $\delta$-equidistributed w.r.t. the measure carried by
\[
(t+i \frac{l^2}{b_{N''}^2+\epsilon_4^2}, -2\arccot  \frac{b_{N''}}{\epsilon_4})    \hspace{20pt}  t\in [0,1].
\]
But, since both $N'$ and $N''$ are in $J$ we have that both measures are similar, and then $s_n$, $n\in J$ is $\delta$-equidistributed w.r.t. $\mu$. 
\end{proof}

As a corollary of Theorem \ref{dani-discrete} and Proposition \ref{discrete-measures} we obtain

\begin{theorem}[Effective equidistribution]\label{effective-equid}
Let $\xi \in X$
fixed. Let $0<\delta <1/2$, $N>1$ and $s>0$.
There exists a positive integer $j<\delta^{-O(1)}$ such that the sequence $\xi h(s)^n$, with $n$ in any subinterval of $n\le N$ of length $N/j$, is $\delta$-equidistributed w.r.t. an algebraic measure on $X$, unless $(sN)^{1/5}s^{-1}<\delta^{-O(1)}$.
\end{theorem}
\begin{remarks}
(a) As shown in the proofs of this section, in the quantitative setting the non-continuous algebraic measures are much more complex than  continuous algebraic ones. Therefore, to be near to a  continuous algebraic measure is a condition that is much stronger than to be near to a general algebraic measure. 

(b) This result gives control over pieces of orbits of the discrete horocycle flow for $s$ not very large. It is possible to show that this control fails for $s>N^{3+\epsilon}$; in fact, taking $x=0$, $y=q^{-2}$, $s=Aq^{-2}$ and $W^{-1}=A^{-2}q^{-3}$  for $A,q$ natural numbers, $q<N^{\epsilon}$, $A>N^{3+\epsilon}$ we have that $gh(s)^n$ is very near to the periodic sequence 
\[
 \frac{n^2}{q}+\frac{i}{q^2}     \hspace{20pt}  n\modu q.
\]
One can show that for certain $q$'s the measure carried by this sequence is not near to any algebraic measure. The same happens for Theorem \ref{dani-discrete}. Perhaps both results remain true for $s<N^{O(1)}$ by substituting algebraic measures by ``polynomial algebraic measures''---and changing the conditions in the statement of Theorem \ref{dani-discrete}---, meaning any measure carried by a periodic sequence $\Gamma h(p(n))a(y)$ with $p$ a polynomial of degree $O(1)$. 
\end{remarks}

\section{Large closure of prime orbits}\label{section-closure-primes}

In this section we are going to deduce Theorem \ref{measure-control} from an upper bound for the sum 
\begin{equation*}
\sum_{p<T} f(xu^p).
\end{equation*}
In order to get such a bound we make use of sieve theory, in particular the following special case of a Selberg's result (see \cite[Theorem~6.4]{iwaniec-kowalski})

\begin{lemma}[Upper bound Sieve]\label{sieve-selberg} Let $(a_n)_{n\le T}$ a sequence of non-negative numbers. For any $d\in \N$ write
\[
\sum_{\substack{n\le T \\ n\equiv 0 \,\mathrm{mod } \, d}} a_n = \frac 1d A + r_d.    
\]
Then, for any $1<D<T$ we have
\[
\sum_{\sqrt{T}< p\le T} a_p \le \frac{A}{\log\sqrt{D}} + \sum_{d< D} \tau_3(d) |r_d|,
\]
with $\tau_3(d)=\sum_{d_1 d_2 d_3 =d} 1$.
\end{lemma}

So the task now is reduced to have tight control over the sums
\[
\sum_{m\le T/d} f(x(u^d)^m)
\]
for most of the $d$'s in a range $1\le d \le D$ with $D$ as big as possible. This can be handled by the  Theorem \ref{dani-discrete} whenever $D<T^{1/5}$. The precise result that follows is 

\begin{theorem}\label{prime-sum}
Let $\xi\in X$ with $\alpha$ irrational, $f\ge 0$ and $s>0$. Then
\begin{equation}\label{PrimeSumBound}
\E_{p<T} f(\xi h(s)^p) \le 10 \int_X f \, d\mu_G  +  o_T(1)\|f\|_{\mathrm{Lip}}.
\end{equation}
\end{theorem}
\begin{proof}
Let $\|f\|_{\Lip}=1$. Let us begin by dealing with the case in which the conditions in Lemma \ref{period-made-explicit-discrete} (i) are satisfied for $\xi$, $s$, $N=T$ and $\delta=(\log T)^{-A}$, $A$ a large constant. Take $q<\delta^{-O(1)}$ the smallest integer satisfying the conditions in the lemma. It is easy to check that $q$ has to go to infinity with $T$;
moreover by Lemma \ref{period-made-explicit-discrete} (ii) and the remarks after the lemma one sees that $\xi h(s)^n$ is at distance $\delta$ from a $q_2$-periodic sequence, with $q_2\mid q[q\frac sR]^2 \ll q^3<\delta^{-O(1)}$, when $n$ is restricted to any subinterval $J$ such that $|J|=\delta^{c}T$ for some constant $c$. In this situation we can apply the Siegel-Walfisz theorem \cite[page 133]{da} for primes in arithmetic progressions  to deduce that
\[
\E_{p<T} f(\xi h(s)^p)=\E_{n<T, (n,q_2)=1} f(\xi h(s)^n) +O((\log T)^{-1}).
\]
The conditions of Lemma \ref{period-made-explicit-discrete} are not satisfied for the parameters $s_*=sd$, $N_*=N/d$, $\delta_*=q^{-\epsilon}$, if $d<q^{\epsilon}$ and $\epsilon>0$ is sufficiently small; we can then apply Theorem \ref{dani-discrete} to deduce that 
\begin{equation}\label{d-independence}
\E_{n<T, n\equiv 0 \text{ mod } d} f(\xi h(s)^n)=\int f \, d\mu +O(q^{-\epsilon}),
\end{equation}
for any $d<q^{\epsilon}$, where $d\mu$ is the average of the algebraic measures appearing in the statement and then it is independent of $d$. Therefore, using the identity $1_{m=1}=\sum_{d\mid m}\mu (d)$ and \eqref{d-independence} for $d<q^{\epsilon}$ we have
\[
\E_{n<T, (n,q_2)=1} f(\xi h(s)^n) = \E_{n<T} f(\xi h(s)^n)+ O(\tau(q_2)q^{-\epsilon})
\]
and then the result follows from Corollary \ref{dani}.

Now let us suppose that the conditions in Lemma \ref{period-made-explicit-discrete} are not satisfied. Then, applying Theorem \ref{dani-discrete} we have
\[
\E_{n<T} f(\xi h(s)^n)=\int f \, d\mu +O((\log T)^{-A}).
\]
with $\mu$ the average of the algebraic measures there. Let us suppose that for any $D\le D_0 =T^{1/5}$ and for any $d$ in $D<d<2D$ but at most $O(\delta D)$ exceptions, the conditions in Lemma \ref{period-made-explicit-discrete} are not satisfied for the parameters $s_*=sd$ and $N_*=N/d$. For any of the non-exceptional $d$  we have  (again by Theorem \ref{dani-discrete})
\[
\E_{n<T, n\equiv 0 \text{ mod } d} f(\xi h(s)^n)=\int f \, d\mu +O((\log T)^{-A}).
\]
Therefore, applying Lemma \ref{sieve-selberg} we get
\[
\E_{p<T} f(\xi h(s)^p)\le 10 \,  \E_{n<T} f(\xi h(s)^n) +O((\log T)^{O(1)-A})
\]
and we are done by Corollary \ref{dani}. So, we have finished unless 
\[
\|d q_d \frac{s}{R_d}\|=D\tau_d^{O(1)}\delta^{-O(1)}T^{-1},  \hspace{16pt}
\|\left[d q_d \frac s{R_d} \right]\alpha_d\|=D\tau_d^{O(1)}\delta^{-O(1)}(sT^2)^{-1}
\] 
for more than $\delta D$ $d$'s in $D<d<2D$ for some $D$, where $R_d=R(\gamma_d g_0)$, $\alpha_d=\alpha(\gamma_d g_0)$, $\Gamma g_0=\xi$, $q_d\ll \tau_d^{O(1)}\delta^{-O(1)}$, and $\tau_d=\tau(\widetilde{(q_d)}_2)$. Let us see that this cannot be true for any $D<T^{1-\epsilon}$, $\epsilon>0$.

Let us assume it is true. We know that $\tau_d< L=D^{o(1)}$. Then we can split $[1,L]$ into at most $O(\log\log L)$  intervals of the shape $[t,2t^2]$, and there will be at least $\delta D (\log\log L)^{-1}\gg \delta^{2} D$ $d$'s with $\tau_d$ in one of them; let us consider now just those $d$'s; for any of them we have $\tau_*\le \tau_d \le \tau_*^2$ for some fixed $\tau_*<L$. 

This implies that $\gamma_d=\gamma$ and $q_d=q$ for a set $\mathcal A$ of more than $D/M$ $d$'s, with $M=\tau_*^{O(1)}\delta^{-O(1)}$. So for them
\[
\|d q \frac{s}{R}\|=DMT^{-1} \, \hspace{30pt}
\|\left[d q \frac s{R} \right]\alpha\|=DM(sT^2)^{-1},
\] 
and then effective equidistribution in the torus (see Lemma 3.2 in \cite{green-tao-equidistribution}) implies that
\begin{equation}\label{torus-equid-sR}
\|hq\frac sR\|<MT^{-1}
\end{equation}
for some $h<M$. We can assume $h$ is coprime to $[hqs/R]$. Now, since $M^{-2}<DMT^{-1}$ it is easy to check that $h\mid d$ for any $d\in \mathcal A$. But then $d=\lambda h$, we have $[d q  s/R]=\lambda[hqs/R]$ and
\[
\|\lambda[hqs/R]\alpha\|<DM(sT^2)^{-1}
\]
for more than $D/M$ $\lambda$'s with $\lambda\ll D $, which again by effective equidistribution in the torus gives
\[
\|j[hqs/R] \alpha\|<M(sT^2)^{-1}
\]
for some $j<M$. Now $j[hqs/R]=[jhqs/R]$ by \eqref{torus-equid-sR}, so choosing $q_*=jq<M$ we have
\begin{equation}\label{TauContradic}
\|q_*\frac{s}{R}\|=MT^{-1} \, \hspace{30pt}
\|\left[q_* \frac sR \right]\alpha\|=M(sT^2)^{-1}.    
\end{equation}
Now, it is easy to check that for any $d\in \mathcal A$ we have
\[
\frac{[[dq\frac sR]\alpha]}{[dq\frac sR]}=\frac{[[q_* \frac sR]\alpha]}{[q_*\frac sR]},
\]
which implies that $\widetilde{(q_d)}_2=k\le [q_*\frac sR] \ll q_*$.  Since $\tau_*<\tau(k)<\tau_*^2$ we arrive at
\[
k\ll q_*<\tau_*^{O(1)}\delta^{-O(1)}<\tau(k)^{O(1)} \delta^{-O(1)}
\]
so $k<\delta^{-O(1)}$ and then $\tau_*<\delta^{-O(1)}$. But then \eqref{TauContradic} means that the conditions in Lemma \ref{period-made-explicit-discrete} are satisfied for the original sequence, which is a contradiction. 
\end{proof}
\emph{Remark:} Theorem \ref{dani-discrete} was not strictly necessary to prove this result (and then Theorem \ref{measure-control}); one could proceed in a more direct way, taking advantage of the extra average in $d$ in Lemma \ref{sieve-selberg}. Anyway, it seems difficult to get a much better level in Lemma \ref{sieve-selberg} that way. We did not do it that way because we think Theorem \ref{dani-discrete} is interesting by itself.

Finally, let us see that Theorem \ref{measure-control} follows from Theorem \ref{prime-sum}. Let $x\in X$ generic, and $\nu_x$ an accumulation point for the sequence $(\pi_{x,N})_N$ in $C^*(X^*)$, where $X^*$ is the one-point compactification of $X$. Take $f\in C(X^*)$, $f\ge 0$. By approximation, we can assume that $f$ as finite Lipschitz norm in $X$, so that Theorem \ref{prime-sum} gives 
\[
\int_X f \, d\nu_x \le 10 \int_X f \, d\mu_G,
\]
and we are done.

\section{Density of the Hecke orbit}\label{density-hecke-orbit}

In this section we are going to prove a stronger result (Theorem \ref{hecke-control}) for the special orbit $H_N h(p)$, $p\le N$, from which we can deduce in particular that it becomes dense in $X$ when $N\to \infty$. This will be possible because we can get a good level of distribution for linear sums, and above all because we can handle bilinear sums up to a considerable level. We input those bounds into the following special case of a sieve result from \cite{duke-friedlander-iwaniec}

\begin{lemma}[Asymptotic sieve]\label{duke-friedlander-iwaniec}
Let $\{a_n\}_{n\in \N}$ be a sequence of non-negative numbers such that $a_n\ll \tau(n)$ and $a_n=A+c_n$ for some constant $A\ge 0$ and a sequence $c_n$ satisfying the ``Type I condition of level $\alpha$''
\begin{equation}\label{TypeI}
\E_{D<d<2D} |\E_{n\le x/d}  \, c_{dn} |\ll (\log x)^{-3}
\end{equation}
for any $D<x^{\alpha-\epsilon}$ and also the ``Type II condition of level $\gamma$''
\begin{equation}\label{TypeII}
\E_{D< d_1 <d_2 <2D} | \E_{n\le \min(x/d_1,x/d_2)} \, c_{d_1 n} \, \overline{c_{d_2 n}}|\ll  (\log x)^{-22},
\end{equation}
for any $D$ with $x^{(\log\log x)^{-3}}\le D \le x^{\gamma-\epsilon} $, for any fixed $\epsilon>0$.
Then we have
\begin{equation}
|\E_{p<x} a_p - A |\le c(\alpha,\gamma)A + O(\epsilon),
\end{equation}
with $c(\alpha,\gamma)$ an explicit decreasing function, such that $c(1/2,1/3)=0$, $c(1/2,1/5)<4/5$ and $c(1/2,\gamma)=1$ for some $\gamma \in (1/6,1/5)$.

Moreover, in the summations we can assume that $d$ is square free and $d_1,d_2$ are primes.
\end{lemma}

Now, we are going to check the Type I condition.

\begin{proposition}[Bound for Type I sums]\label{hecke-type-i}
Let $f$ with $\int_X f=0$. We have that
\[
\E_{D<d<2D}|\E_{n\le N/d} \, f(H_N h(dn))| \ll (\log N)^{-3}\|f\|_{\mathrm{Lip}}
\]
for any $D<N^{1/2-\theta-\epsilon}$, for any  $\epsilon>0$.
\end{proposition}
\begin{proof}
Let us assume $\|f\|_{\Lip}=1$, and suppose the result is false. Then, we have
\[
\E_{D <d <2D} |\E_{n} f(\Gamma H_N h(dn))\eta(\frac{Dn}{N})|\gg \delta
\]
for $\delta=(\log N)^{-A}$, $A$ a large constant, with $\eta \in C_0^{\infty}(-2,2)$ with $\|\eta^{(j)}\|_{L^{\infty}}, \|f\|_{W^j}\ll_j \delta^{-j}$. We have the Iwasawa parametrization
\[
H_N h(dn)=h(dnN^{-1})a(N^{-1}),
\]
and then by the Fourier expansion $f(h(x)a(y))=\sum_m \tilde f (m, y) e(mx)$ and Poisson Summation we have
\[
\E_{D<d<2D}\sum_m |\tilde f(m,N^{-1})| |\hat\eta(\frac ND \{\frac{dm}N \})| \gg \delta
\]
so that from Proposition \ref{period-bound-proposition} we have
\[
\delta^{-O(1)} D^{-1}N^{-1/2+\theta_2} \sum_{j\ll \delta^{-1} DN} \tau(j)^2|\hat \eta(\frac ND \{\frac jN \} )|\gg 1
\]
which taking into account the decay of $\hat\eta$ gives a contradiction.
\end{proof}

Let us go with the Type II condition

\begin{proposition}[Bound for Type II sums]\label{hecke-type-ii}
Let $f_1, f_2$ be continuous functions in $\Gamma\backslash G$ with $\|f_1\|_{\Lip}=\|f_2\|_{\Lip}=1$ and $\int_X f_1=0$ . Let $N^{(\log\log N)^{-3}}<D<N^{(1-2\theta)/(5+2\theta)-\epsilon}$ and $D<d_1 < d_2 <2D$, with $d_1, d_2$ primes. Then we have
\[
\mathbb E_{n\le \min(N/d_1,N/d_2)} f_1(H_N h(d_1 n)) f_2(H_N h(d_2 n)) \ll (\log N)^{-22}.
\]
\end{proposition}
\begin{proof}
As in (4.1) we can replace the Lipschitz norm by Sobolev norms and hence if the result is false, we have
\[
|\mathbb E_{n< N/D} f_1(H_N h(d_1 n)) f_2(H_N h(d_2 n)) \eta(\frac{Dn}{N}) |\gg \delta
\]
for $\delta=(\log N)^{-A}$, $A$ a large constant, with $\eta \in C_0^{\infty}(0,1)$ with $\|\eta^{(j)}\|_{L^{\infty}}\ll\delta^{-j}$ and $\|Df_i\|_{L^{\infty}}\ll_{\mathrm{ord}D}\delta^{-\mathrm{ord} D}$. By the Fourier expansion of $f_i$, $f_i(h(x)a(y))=\sum_{m} \tilde f_i(m,y) e(mx)$, the bounds in Proposition \ref{period-bound-proposition},
Poisson summation in $n$ and integration by parts we have that either 
\begin{equation*} 
|\sum_{d_1 m_1+d_2 m_2=0} \tilde f_1(m_1,N^{-1}) \tilde f_2 (m_2,N^{-1})|\gg \delta
\end{equation*}
or 
\begin{equation}\label{second-possibility}
|\sum_{d_1 m_1+d_2 m_2=k} \tilde f_1(m_1, N^{-1}) \tilde f_2 (m_2,N^{-1})|\gg \delta D^{-2}
\end{equation}
for some $k\neq 0$, $k\ll \delta^{-O(1)} DN$. If the first possibility is true, we get
\[
|\sum_{j} \tilde f_1(d_1 j,N^{-1}) \tilde f_2(d_2 j,N^{-1})|\gg \delta.
\]
But from Proposition \ref{period-bound-proposition} we have
$
\tilde f_i (0,N^{-1})\ll \delta ^{-O(1)}N^{-\frac 12},
$
and from the spectral expansion (\ref{spectral-decom-func}) of $f_i$,
the multiplicativity of Hecke eigenvalues (\ref{multiplicativity}) and Parseval \eqref{parseval-discrete} we have
\begin{equation}\label{k-zero}
\sum_{j\neq 0} \tilde f_1(d_1 j,N^{-1}) \tilde f_2(d_2 j,N^{-1})\ll (d_1 d_2)^{\theta-\frac 12},
\end{equation}
which gives a contradiction.

Let us now assume that \eqref{second-possibility} is true. This is a shifted convolution sum, and we can proceed as in \cite{blomer-harcos,blomer-harcos-adeles}. We can translate it as
\[
|\int_0^1 f_1(h(d_1 x)a(\frac 1N))f_2(h(d_2 x)a(\frac 1N)) e(-kx) \, dx|\gg \delta D^{-2}, 
\]
and further as
\[
|\int f_*(h(x) a(\frac 1{DN})) e(-kx) \, dx | \gg \delta D^{-2}, 
\]
with $f_*=f_{d_1} f_{d_2}$ and
\[
f_{d}(g)= f( a(d)\,  g  \, a(D/d) ).
\]
$f_{d}$ is a $\Gamma_0(d)-$invariant function and since $D/d \asymp 1$ we have
$\|Df_{d}\|_{L^{\infty}}\ll \|Df\|_{L^{\infty}}$; thus, $f_*$ can be seen as a function in $\Gamma_0(d_1d_2)\backslash G$ with Sobolev norms 
$\|f_*\|_{W^j} \ll_j (d_1 d_2)^{\frac 12 +\epsilon}\delta ^{-j}$.
Therefore, from Proposition \ref{period-bound-proposition} we get the bound
\[
\tilde f_*(k, (DN)^{-1})\ll \delta^{-O(1)} (DN)^{-\frac 12+\theta}(d_1 d_2)^{\frac 12 +\epsilon}
\]
which gives a contradiction in our range for $D$.
\end{proof}

Now, assuming that $\theta=0$, due to Propositions \ref{hecke-type-i} and \ref{hecke-type-ii} we can apply Lemma \ref{duke-friedlander-iwaniec} with $\alpha=1/2$ and $\gamma=1/5$ for  $a_n=f(\Gamma H_N h(n))$, and then

\begin{theorem}\label{hecke-almost-equidistribution}
Assuming $\theta=0$, for any non-negative $f$ we have
\[
|\E_{p<N} f(H_N h(p))-\int_X f \, d\mu_G |\le \frac 45 \int_X f \, d\mu_G +o(1) \|f\|_{\Lip}
\]
\end{theorem}

From this result we can deduce Theorem \ref{hecke-control} as we did in the previous section with Theorem \ref{measure-control}.

We end by discussing some improvements on the levels of distribution in Propositions \ref{hecke-type-i} and \ref{hecke-type-ii}. For the type I sums we will establish a level of $1/2$ matching what Proposition \ref{hecke-type-i} gives with $\theta=0$. For the type II sums, Proposition \ref{hecke-type-ii} with $\theta= 7/64$ yields a level of $25 / 167 = 0.1457\ldots$. We establish a level of $3/19=0.1578\ldots$ which is a small improvement but still falls short of the magic number $c$ $(\frac 16 < c < \frac 15  )$ which would make Theorem \ref{hecke-almost-equidistribution} unconditional.

\begin{proposition}\label{referee}
The result in Proposition \ref{hecke-type-i} is true for any $D<N^{1/2-\epsilon}$, for any $\epsilon>0$. 
\end{proposition} 
\begin{proof}
In what follows we assume that $f$ is orthogonal to the Eisenstein series, because for them we have $\theta=0$; moreover, for simplicity let us consider $f$ $K$-invariant. The sums in Proposition \ref{hecke-type-i} are (assuming we have a smooth sum, as we can)
\[
I   =\frac 1N  \sum_{d\asymp D} |\sum_n \eta(\frac{Dn}N) f(H_N h(dn))|.
\]
We can also assume that $f$ is orthogonal to the space of Eisenstein series (because for them we know that $\theta=0$), and then we can write the Fourier expansion 
\[
f(h(x)a(\frac 1N))=\sum_k \tilde f(k, \frac 1N) e(kx)=O(\delta)+ \sum_{k}^*  \tilde f(k, \frac 1N) e(kx)
\]
with $\delta=N^{-1/\log\log N}$ and  the sum in $k$ restricted to $|k|<\delta^{-O(1)} N$ and $(k,N)<\delta^{-O(1)}$. This is done by using the spectral expansion, the multiplicativity of $\lambda_{\pi}(k)$ and the bounds \eqref{Peter} and \eqref{bound-fourier-series-discrete}. Thus it is enough to bound
\[
I_*=\frac 1N \sum_{d\asymp D} |\sum_k^* \tilde f (k, \frac 1N) \sum_n \eta(\frac{Dn}N) e(\frac{dkn}N) |. 
\]
Applying Poisson Summation in $n$ yields
\begin{equation}
I_*\ll \frac 1D \sum_{d\asymp D} \sum_{k,\nu}^* |\tilde f(k, N^{-1})| |\widehat{\eta} (\frac{dk-\nu N}D ) | \ll \frac 1D \sum_k^* c_k |\tilde f (k,\frac 1N)|,
\end{equation}
with $c_k$ the number of $d, t\ll D$ such that $dk\equiv t \modu N$. 
For $D\le N^{\frac 12-\epsilon}$ which we assume is in force, one can check that for each $k$ as above $c_k\ll N^{\epsilon}$ as follows: if $t=0$ there are no solutions for the congruence since $(k,N)<\delta^{-O(1)}$; otherwise, for any pair of solutions $(d_1,t_1)$ and $(d_2,t_2)$ we have that $d_1 t_2\equiv d_2 t_1 \modu N$, so that $d_1 t_2 =d_2 t_1$. This means that any solution is of the form $(d,t)=\lambda (d_*, t_*)$, for some $(d_*,t_*)$ fixed; since $\lambda$ divides $N$ we are done.
On the other hand, we have $\sum_k^* c_k \ll_{\epsilon} D^2 N^{\epsilon}$, so that $\sum_k^* c_k^m \ll_{\epsilon, m} D^{2}N^{\epsilon}$.

From (\ref{period-spectral-expression}) and noting that $q=1$ we have 
\[
|\tilde f (k, N^{-1}) \ll_{\epsilon} \delta + N^{-\frac 12+\epsilon} \sum_{|t_j|\ll \delta^{-O(1)}} |\lambda_j (k)|
\]
and hence
\begin{equation}
I_*\ll \delta+ \frac{N^{-\frac 12 +\epsilon}}D \sum_j \sum_{k}^* c_k|\lambda_j (k)| \ll \frac{N^{-\frac 12+\epsilon}}D \sum_{j} (D^{2}N^{\epsilon})^{\frac 34}  (\sum_{k}^* |\lambda_j (k)|^4 )^{\frac 14} .
\end{equation}
It is known (see \cite{kim-sarnak} and \cite{li}) that for $x\ge 1$, 
\begin{equation}\label{fourth-moment}
\sum_{|k|\le x} |\lambda_j(k)|^4 \ll_{\epsilon} \lambda_j^{\epsilon} x,
\end{equation}
hence
\begin{equation}
I_*\ll \delta+\delta^{-O(1)}\frac{N^{-\frac 12 +\epsilon}}D N^{\frac 14} D^{\frac 32} = \delta +\delta^{-O(1)}N^{-\frac 14+\epsilon } D^{\frac 12}.
\end{equation}
This gives a level of distribution of $1/2$ for these type I sums.
\end{proof}
\begin{proposition}
The result in Proposition \ref{hecke-type-ii} is true in the range for $N^{(\log\log N)^{-3}}<D<N^{3/19-\epsilon}$ for any $\epsilon>0$.
\end{proposition} 
\begin{proof}
 We continue with the setting described in the proof of Proposition \ref{referee}. For the type II sums $y^{-1}=ND$, $q=D^2$ and we could try to use the improvement in Proposition \ref{period-bound-proposition} in the range $\sqrt q y^{-\theta} > y^{-\frac 14}$, i.e. $D(ND)^{\theta}>(ND)^{\frac 14}$ or $D>N^{\frac{1-4\theta}{3+4\theta}}$. For $\theta=7/64$ this is $D>N^{9/55}$. However $\frac 9{55}> \frac{25}{167}$ so that the improvement in this range does not give an improvement of the level $25/167$. Instead we again exploit the average over $k$ and again we use \cite{li}; for $x\ge 1$
\begin{equation}\label{eight-moment}
\sum_{|m|\le x} |\lambda_j (m) |^8 \ll_{\epsilon} (\lambda_j q)^{\epsilon} x.
\end{equation}
For $d_1, d_2 \asymp D$  
\begin{align}
II & = \frac DN \sum_n \eta(\frac{Dn}N) f_1(H_N h(d_1 n)) f_2(H_N h(d_2 n))   \\
& = \frac DN \sum_n \eta(\frac{Dn}N) f_* (\frac nN, \frac 1{ND} ) \notag \\
& \ll D^{-\epsilon}+ \sum_{|k|\ll ND^{1+O(\epsilon)}} \widetilde{f_*} (k, (ND)^{-1}) \sum_{|\nu|\ll D^{1+\epsilon}} \widehat{\eta} (\frac{k-\nu N}D) \notag  \\
& \ll D^{-\epsilon}+ \sum_{|t|, |\nu|\le D^{1+O(\epsilon)}} |\widetilde{f_*} (t+\nu N, (ND)^{-1})|. \notag
\end{align}
From (\ref{period-spectral-expression}) with $|k|\ll ND^{1+\epsilon}$ and for $k=0$ due to (\ref{k-zero}) we have that
\[
|\widetilde{f_*}(k, (ND)^{-1})|\ll D^{-\epsilon}+ \frac{(ND)^{-\frac 12}}{\sqrt q}\sum_{|t_j|\ll D^{O(\epsilon)} }|\lambda_j(k)| |\langle f_*, \phi_j\rangle |.
\]
Hence
\begin{equation}\label{average-periods}
\sum_{|t|,|\nu |\le  D^{1+\epsilon}} |\widetilde{f_*} (t+\nu N, (ND)^{-1}) | \le   \frac{(ND)^{-\frac 12}}{\sqrt q}\sum_{|t_j|\ll D^{\epsilon}} |\langle f_*,\phi_j \rangle | \sum_{|t|, |\nu |\le D^{1+\epsilon}} |\lambda_j (t+\nu N) |.
\end{equation}
The inner sum may be estimated by Holder,
\[
\sum_{|t|,|\nu |\le D^{1+\epsilon}} |\lambda_j (t+\nu N) | \le (\sum_{|m|\le ND^{1+\epsilon}} |\lambda_j (m) |^8 )^{\frac 18} \, (D^{2+2\epsilon} )^{\frac 78}
\]
which by (\ref{eight-moment}) is
\begin{equation}
\ll \lambda_j^{\epsilon} (ND)^{\frac 18} D^{\frac 74}D^{O(\epsilon)}.
\end{equation}
Substituting this into (\ref{average-periods}) gives (recall that $q=D^2$) 
\begin{align}
II & \ll D^{-\epsilon}+ \frac{(ND)^{-\frac 12 +\epsilon}}{\sqrt q} (ND)^{\frac 18} D^{\frac 74} (\sum_{|t_j|\ll D^{\epsilon}} |\langle f_*, \phi_j\rangle |^2 )^{\frac 12} \, q^{\frac 12} \\
& \ll D^{-\epsilon}+(ND)^{-\frac 12} (ND)^{\frac 18} D^{\frac 74} (D^2)^{\frac 12} N^{O(\epsilon)} =D^{-\epsilon}+N^{-\frac 38} D^{\frac{19}8} N^{O(\epsilon)} \notag.
\end{align}
This gives a level of distribution of $3/19$ for the type II sums.
\end{proof}

\

\begin{center}
\textsc{Acknowledgments} 
\end{center}
We would like to thank N. Pitt for interesting discussions about his paper \cite{pi}. We also acknowledge the referees of an earlier version of this paper for pointing to some errors that needed correcting.

A. Ubis was supported by a Postdoctoral Fellowship from the Spanish Government during part of the writing of the paper; he also thanks both the Department of Mathematics of Princeton University and the Institute for Advanced Study for providing excellent working conditions. P. Sarnak and A. Ubis were supported in part by NSF grants, and A. Ubis by a MINCYT grant.

\end{document}